\documentclass[12pt,a4paper]{article}

\usepackage{amssymb}
\usepackage{amsmath, amsthm}
\usepackage{arydshln}
\usepackage{epsfig}
\usepackage{setspace}

\usepackage{geometry}

\def\RR{{\bf R}}
\def\ZZ{{\bf Z}}

\numberwithin{equation}{section}

\newcommand{\rank}{\mathop{\rm rank} }
\newcommand{\diam}{\mathop{\rm diam} }

\newcommand{\argmin}{\mathop{\rm argmin} }

\newtheorem{Thm}{Theorem}[section]
\newtheorem{Prop}[Thm]{Proposition}
\newtheorem{Lem}[Thm]{Lemma}
\newtheorem{Cor}[Thm]{Corollary}
\theoremstyle{definition}
\newtheorem*{Clm}{Claim}

\newtheorem{Rem}[Thm]{Remark}

\title{Maximum vanishing subspace problem, CAT(0)-space relaxation,
	\\ and block-triangularization of partitioned matrix
	\footnote{An extended abstract of this paper appears in the proceeding of 10th Japanese-Hungarian Symposium
		on Discrete Mathematics and Its Applications}}
\author{Masaki HAMADA and Hiroshi HIRAI \\
Department of Mathematical Informatics, \\
Graduate School of Information Science and Technology,   \\
The University of Tokyo, Tokyo, 113-8656, Japan.\\
\texttt{\normalsize masaki$\_$hamada@mist.i.u-tokyo.ac.jp}\\
\texttt{\normalsize hirai@mist.i.u-tokyo.ac.jp}
}
\begin{document}
\maketitle
\begin{abstract}
In this paper, we address the following 
algebraic generalization of the bipartite stable set problem.
We are given a block-structured matrix (partitioned matrix) $A = (A_{\alpha \beta})$, where 
$A_{\alpha \beta}$ is an $m_{\alpha}$ by $n_{\beta}$ matrix 
over field ${\bf F}$ for $\alpha=1,2,\ldots,\mu$ and $\beta = 1,2,\ldots,\nu$.  
The maximum vanishing subspace problem (MVSP) 
is to maximize $\sum_{\alpha} \dim X_{\alpha} + \sum_{\beta} \dim Y_{\beta}$ 
over vector subspaces 
$X_{\alpha} \subseteq {\bf F}^{m_{\alpha}}$ for  $\alpha=1,2,\ldots,\mu$ and $Y_{\beta} \subseteq {\bf F}^{n_{\beta}}$ for $\beta = 1,2,\ldots,\nu$
such that each
$A_{\alpha \beta}$  vanishes on $X_{\alpha} \times Y_{\beta}$ when 
$A_{\alpha \beta}$ is viewed as a bilinear form ${\bf F}^{m_{\alpha}} \times {\bf F}^{n_{\beta}} \to {\bf F}$.
This problem arises from a study 
of a canonical block-triangular form of $A$ 
by Ito, Iwata, and Murota~(1994), and is closely related 
to the noncommutative rank of a matrix with indeterminates. 

We prove that a weighted version (WMVP) of MVSP can be solved in psuedo polynomial time, 
provided arithmetic operations on ${\bf F}$ can be done in constant time. 
Our proof is a novel combination of submodular optimization on modular lattice
and convex optimization on CAT(0)-space.
We present implications of 
this result on block-triangularization of partitioned matrix.
\end{abstract}


%MSC classes: 68Q25, 68Q25, 15A21,  15B99.

\section{Introduction}
The maximum stable set problem in bipartite graphs is one of the fundamental 
and well-solved combinatorial optimization problems.
We address in this paper the following algebraic generalization 
of the bipartite stable set problem.
We are given a matrix $A$ partitioned into submatrices as
\[
A = \left(
\begin{array}{ccccc}
A_{11} & A_{12} &\cdots & A_{1\nu} \\
A_{21} & A_{22} &\cdots & A_{2\nu} \\
\vdots & \vdots & \ddots & \vdots \\
A_{\mu1}&A_{\mu2} &\cdots & A_{\mu \nu}
\end{array}\right),
\]
where $A_{\alpha \beta}$ is an $m_\alpha \times n_\beta$ 
matrix over field ${\bf F}$ for $\alpha=1,2,\ldots,\mu$ and
$\beta=1,2,\ldots,\nu$. 
Such a matrix is called a {\em partitioned matrix of 
	type $(m_1,m_2, \ldots, m_\mu; n_1,n_2,\ldots,n_\nu)$}. 
The {\em maximum vanishing subspace problem (MVSP)} is to maximize
\begin{equation}\label{eqn:objective}
\sum_{\alpha=1}^\mu \dim X_{\alpha} + \sum_{\beta=1}^\nu \dim Y_{\beta} 
\end{equation}
over vector subspaces $X_\alpha \subseteq {\bf F}^{m_{\alpha}}$ for $\alpha =1,2,\ldots,\mu$ and
$Y_\beta \subseteq {\bf  F}^{n_{\beta}}$ for $\beta =1,2,\ldots,\nu$ satisfying
\begin{equation}\label{eqn:vanishing}
A_{\alpha \beta} (X_{\alpha}, Y_{\beta}) = \{0\} \quad 
(1 \leq \alpha \leq \mu, 1 \leq \beta \leq \nu),
\end{equation}
where each submatrix $A_{\alpha \beta}$ is regarded 
as a bilinear form ${\bf F}^{m_{\alpha}} \times {\bf F}^{n_{\beta}} \to {\bf F}$ by
\begin{equation*}
(u,v) \mapsto u^{\top} A_{\alpha \beta} v.
\end{equation*}
A tuple $(X_{1}, X_2,\ldots,X_{\mu},Y_1,Y_2,\ldots, Y_{\nu})$ satisfying (\ref{eqn:vanishing}) 
is called a {\em vanishing subspace}, 
and is called {\em maximum}  
if it has the maximum dimension, where the {\em dimension} is defined as (\ref{eqn:objective}).  

MVSP generalizes 
the maximum stable set problem on bipartite graphs.
Indeed, consider the case $m_{\alpha} = n_{\beta} = 1$ for each $\alpha, \beta$. Namely each submatrix is a scalar.
 Then each vector subspace is $\{0\}$ or ${\bf F}$, 
 and its dimension is $0$ or $1$. 
 The condition~(\ref{eqn:vanishing}) 
 says that  one of $X_{\alpha}$ and $Y_{\beta}$ is $\{0\}$ if $A_{\alpha \beta}$ is a nonzero scalar.
Thus MVSP is the maximum stable set problem on 
a bipartite graph on vertices 
$a_1,a_2,\ldots,a_\mu, b_{1},b_2,\ldots b_\nu$ such that edge 
$a_{\alpha} b_{\beta}$ is given if and only if $A_{\alpha \beta}$ is a nonzero scalar.  

A linear algebraic interpretation of MVSP is explained as follows.
Consider a transformation of $A$ of the form
\begin{equation}\label{eqn:trans}
%P^{\top} 
\left(
\begin{array}{cccc}
E_1^{\top} & O  & \cdots & O  \\
O & E_2^{\top} & \ddots & \vdots \\
\vdots & \ddots & \ddots &   O\\
O & \cdots & O& E_{\mu}^{\top}
\end{array}
\right)
\left(
\begin{array}{ccccc}
A_{11} & A_{12} &\cdots & A_{1\nu} \\
A_{21} & A_{22} &\cdots & A_{2\nu} \\
\vdots & \vdots & \ddots & \vdots \\
A_{\mu1}&A_{\mu2} &\cdots & A_{\mu \nu}
\end{array}\right)
\left(
\begin{array}{cccc}
F_1 & O  & \cdots & O  \\
O & F_2 & \ddots & \vdots \\
\vdots & \ddots & \ddots &   O\\
O & \cdots & O& F_{\nu}
\end{array}
\right),
%Q,
\end{equation}
where $E_\alpha$ is a nonsingular $m_\alpha \times m_\alpha$ matrix for $\alpha = 1,2,\ldots,\mu$ and 
$F_\beta$ is a nonsingular $n_\beta \times n_\beta$ matrix 
for $\beta = 1,2,\ldots, \nu$.
If the resulting matrix contains 
a zero submatrix of $c$ rows and $d$ columns,
then from $E_{\alpha}$ and $F_{\beta}$
we obtain a vanishing subspace of dimension $c+d$.
Conversely, from a vanishing subspace 
of dimension $b$,  
we can find a transformation of form (\ref{eqn:trans})
such that the resulting matrix contains 
a zero submatrix of $c$ rows and $d$ columns with $c+d = b$. 
Thus MVSP is nothing but the problem of finding 
a transformation (\ref{eqn:trans}) of $A$ such that 
the resulting matrix has the largest zero submatrix.

Ito, Iwata, and Murota~\cite{ItoIwataMurota94} 
studied a canonical block-triangular form under transformation~(\ref{eqn:trans}), 
which generalizes 
the classical {\em Dulmage-Mendelsohn decomposition}~\cite{DulmageMendelsohn58}; 
see also \cite{LovaszPlummer}.
They formulated an equivalent problem of MVSP, though
MVSP was explicitly introduced by a recent paper~\cite{HH16DM}.
For several basic special cases~\cite{DulmageMendelsohn58,HH16DM,MurotaIriNakamura87}, 
MVSP can be solved in polynomial time via Gaussian elimination, bipartite matching, and matroid intersection algorithm,  
and a canonical block-triangular form is also obtained accordingly.
These works are in a cross road of numerical computation 
and combinatorial optimization. 
Ito, Iwata, and Murota~\cite[p.1252]{ItoIwataMurota94} 
raised an open problem of solving  
(an equivalent problem of) MVSP and obtaining a canonical block-triangular form 
in polynomial time.
%
%The main result of this paper solves the former part of this problem.

The contribution of this paper is about this problem.  
We consider a natural weighted generalization of MVSP.
We are further given nonnegative weights $C_{\alpha},D_{\beta}$ 
for $1 \leq \alpha \leq \mu$ and $1 \leq \beta \leq \nu$.
The {\em weighted maximum vanishing subspace problem (WMVSP)}
asks to maximize 
\begin{equation*}
\sum_{\alpha}C_{\alpha} \dim X_{\alpha} + \sum_{\beta} D_{\beta} \dim Y_{\beta} 
\end{equation*}
over all vanishing subspaces 
$(X_{1},X_{2},\ldots,X_{\mu},Y_{1},Y_2,\ldots,Y_{\nu})$.
Let $m := \sum_{\alpha} m_{\alpha}$ and $n := \sum_{\beta} n_{\beta}$, i.e.,  $A = (A_{\alpha \beta})$ is an $m \times n$ matrix. 
The main result is the pseudo-polynomial time solvability of WMVSP:
\begin{Thm}\label{thm:main'}
	Suppose that arithmetic operations on ${\bf F}$ can be done in constant time.
	WMVSP can be solved in time polynomial in
	$m,n$ and $W$, where $W$ is the maximum of weights $C_{\alpha}, D_{\beta}$.
\end{Thm}
The algorithm in this theorem is applicable to the case where ${\bf F}$ is a finite field (with $\log |{\bf F}|$ fixed).
However, if ${\bf F}$ is a rational field ${\bf Q}$, then 
the required bit length is not bounded, 
though our algorithm only requires a polynomial number 
of arithmetic operations.

%%
%\begin{Thm}\label{thm:main}
%MVSP can be solved in polynomial time.
%\end{Thm}
%
Significances, implications, and novel proof techniques of this result 
are explained in the following.

\paragraph{Submodular optimization on modular lattice.}
MVSP (or WMVSP) is viewed as a submodular function minimization (SFM)
on the lattice of all vector subspaces of a vector space.
Such a lattice is a typical instance of a {\em modular lattice}.
Submodular optimization on modular lattice 
is a challenging field in combinatorial optimization.
Kuivinen~\cite{Kuivinen09,Kuivinen11} proved a good characterization
of SFM on the product ${\cal L}^n$ of a modular lattice ${\cal L}$, 
where ${\cal L}$ is finite and is a part of the input.
In this setting, 
Fujishige, Kir\'aly,  Makino, Takazawa, and Tanigawa~\cite{FKMTT14} proved 
the polynomial time solvability of SFM on ${\cal L}^n$ 
where ${\cal L}$ is a modular lattice of rank $2$.
In the valued-CSP setting where 
a submodular function is given as a sum of submodular functions with a fixed number of variables, 
the tractability criterion of Kolmogorov, Thapper, and \v{Z}iv\'ny~\cite{KTZ13} 
implies that SFM on ${\cal L}^n$ is solvable in polynomial time.
In contrast with these results, our SFM is defined
on an {\em infinite} modular lattice ruled by a linear algebraic machinery.
Our result is a step toward understanding this type of discrete optimization problems over a lattice of vector subspaces.
%To the best of our knowledge, 
%Theorem~\ref{thm:main} is 
%the first positive result 
%on this type of discrete optimization problems 
%over an infinite lattice of vector subspaces.

\paragraph{Beyond Euclidean convexity: Outline of the proof.}
No reasonable LP/convex relaxation  (allowing infiniteness) 
is known for MVSP and WMVSP.
This is a main reason of the difficulty. 
Beyond Euclidean convexity, 
our proof employs a method of 
a {\em non-Euclidean convex optimization}, 
more specifically, {\em convex optimization on CAT(0)-space.}
Here 
a {\em CAT(0)-space} is a nonpositively-curved metric space enjoying various fascinating properties analogous to those in the Euclidean space; see \cite{BrHa}.
One of important features of a CAT(0)-space 
is the unique geodesic property: every pair of points 
can be joined by a unique geodesic.
Through the unique geodesics, several convexity 
concepts (e.g., convex functions) are naturally introduced.
Computational and algorithmic theory on CAT$(0)$-space 
is another challenging research field; see e.g., \cite{ArdilaOwenSullivant12,Bacak13,Owen11}.
Our proof explores the power of the convexity of CAT$(0)$-space
to obtain the polynomial time complexity in discrete optimization. 

As is well-known, a (usual) submodular function 
on Boolean lattice $\{0,1\}^n$
is extended to a convex function on hypercube $[0,1]^n$ in the Euclidean space, 
via {\em Lov\'asz extension}~\cite{Lovasz83}. 
This fact enables us to apply Euclidean convex optimization methods 
(e.g., the ellipsoid method) to various 
problems related to submodular functions. 
Analogous to the embedding $\{0,1\}^n \hookrightarrow [0,1]^n$, 
a modular lattice ${\cal L}$ is embedded into 
a suitable continuous metric space $K({\cal L})$, called 
the {\em orthoscheme complex}~\cite{BM10}. 
Figure~\ref{fig:folder} illustrates the orthoscheme complex of a modular lattice of rank 2, which is obtained by 
gluing Euclidean triangles along one common side.  
	\begin{figure}[t]
		\begin{center}
			\includegraphics[scale=0.5]{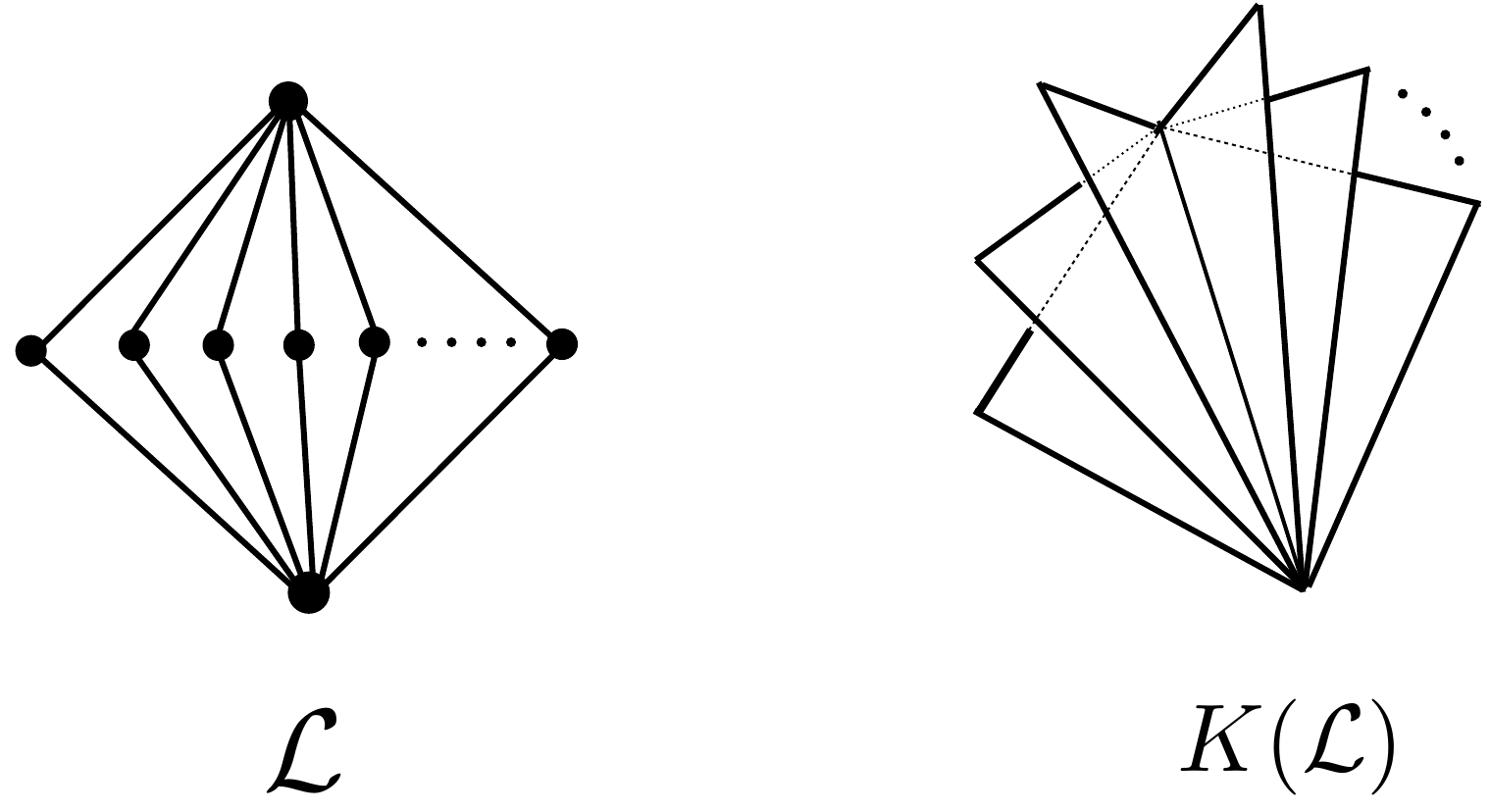}
			\caption{Modular lattice ${\cal L}$ and orthoscheme complex $K({\cal L})$}
			\label{fig:folder}
		\end{center}
	\end{figure}\noindent
It is shown in~\cite{CCHO,HKS17} that $K({\cal L})$ is a
CAT$(0)$-space.
In this setting,
a submodular function is extended to a convex function 
on $K({\cal L})$~\cite{HH16L-convex}.
Consequently, our problem WMVSP becomes 
a convex optimization over a CAT$(0)$-space.

We solve this continuous optimization problem 
by using a CAT(0)-space version
of a {\em proximal point algorithm (PPA)}.
The Euclidean PPA is a well-known simple iterative algorithm 
to minimize a convex function $f$, which
computes the proximal point operator $J_{\lambda}^{f}(z)$ 
of the current point $z$, 
updates $z \leftarrow J_{\lambda}^f(z)$, and repeat.
The PPA is naturally defined on a CAT(0)-space.
Ba\v{c}\'ak~\cite{Bacak13} showed that 
the sequence $(z_{\ell})$ generated by PPA
converges to a minimizer of $f$; see also \cite{BacakBook}.
We apply a version of PPA to our CAT(0)-space relaxation of WMVSP.
By using a recent result of Ohta and P\'alfia~\cite{OhtaPalfia15} 
on the rate of the convergence,   
we show that after a polynomial number of iterations,
a maximum vanishing space is obtained from the current point $z_{\ell}$.
We prove that the proximal operator in each step 
is computed in polynomial time. 
This is the most technical but intriguing part of the proof.

\paragraph{Block-triangularization of partitioned matrix.}
Let us return to the original motivation of MVSP. 
A maximal chain of the maximum vanishing subspaces
provides, via an appropriate change of bases, 
the most refined block-triangularization under transformation (\ref{eqn:trans}), 
which we call the {\em DM-decomposition}~\cite{HH16DM,ItoIwataMurota94}. 
%(since it generalizes the Dulmage-Mendelsohn decomposition). 
Solving MVSP is not enough to obtaining the DM-decomposition.
We here introduce a reasonably {\em coarse} block-triangularization, 
which we call a {\em quasi DM-decomposition}.
A quasi DM-decomposition still generalizes known important special cases,  
such as CCF for mixed matrices~\cite{MurotaIriNakamura87}.
We show that a quasi DM-decomposition
can be obtained in polynomial time
by solving WMVSP with varying weights.
We believe that obtaining
a quasi DM-decomposition is a limit which we can do by 
combinatorial or optimization methods.
A step to DM-decomposition from quasi DM-decomposition
seems to be a matter of numerical analysis/computation, and includes
the common invariant subspace problem, 
which is an extremely difficult numerical computational problem (see e.g.,\cite{ArapuraPeterson04,JamiolkowskiaPastuszakb15}).

\paragraph{Relation to Edmonds' problem and the recent development.}
After finishing the first version of this paper, 
we found that our result is closely related to 
the recent remarkable development~\cite{DerksenMakam17,GGOW15,GGOW16,IQS15a,IQS15b} 
on Edmonds' problem.
We briefly explain this fact.
{\em Edmonds$'$ problem}~\cite{Edmonds67} asks, 
given a vector space ${\cal A}$ 
of $m \times n$ matrices, 
to determine the maximum rank over matrices in ${\cal A}$.
Suppose that the matrix space ${\cal A}$ 
is given as its basis $A_1,A_2,\ldots A_N$.
Then the problem is to determine
\begin{equation}\label{eqn:rank}
\rank {\cal A} := \max \left\{\rank \sum_{i=1}^N z_i A_i \mid z_i \in {\bf F} \ (i=1,2,\ldots,N) \right\}. 
\end{equation}
As is noticed in~\cite{Lovasz89}, there is a weak duality.  
For a nonnegative integer $c$, a {\em $c$-shrunk subspace} 
is a vector subspace $X \subseteq {\bf F}^m$ with 
\[
\dim X - \dim \sum_{i=1}^{N} A_i (X) \geq c,
\]
where $A_i$ is viewed as ${\bf F}^{m} \to ({\bf F^*})^n$ 
by $x \mapsto x^\top A_i$ for dual ${\bf F}^*$ of ${\bf F}$.
For a $c$-shrunk subspace $X$,
via a basis transformation
(including bases of $X$ and $\sum_{i=1}^{N} A_i (X)$), 
all $A_i$ (i.e., all matrices in ${\cal A}$)
are transformed to have the zero block of size 
$n + c$ in the same position.
Consequently, $\rank {\cal A}$ is bounded by $m - c$.
The relation between 
shrunk subspaces and vanishing subspaces is explained as follows, 
where a vanishing subspace $(X,Y)$ in this setting is meant as 
a pair $(X,Y)$ of subspaces in ${\bf F}^m$ and ${\bf F}^n$
with $A_i(X,Y) = \{0\}$ for all $i$.
For a $c$-shrunk subspace $X$,
we obtain a vanishing subspace $(X,Y)$ with $\dim X + \dim Y \geq n + c$ for $Y := (\sum_{i=1}^{N} A_i (X))^{\bot} = \bigcap_{i=1}^N (A_i(X))^{\bot}$,
where $(\cdot)^{\bot}$ means the orthogonal complement.  
Conversely,
$X$ for a vanishing subspace $(X,Y)$ is a $(\dim X + \dim Y - n)$-shrunk subspace.
Summarizing, we obtain the following weak duality relation:
\begin{eqnarray}\label{eqn:weakdual1}
\rank {\cal A} &\leq &\min \{ m - c \mid c \geq 0: \mbox{$\exists c$-shrunk subspace $X$}\} \\
& = & \min \{m + n - \dim X - \dim Y \mid \mbox{vanishing subspace $(X,Y)$} \}.\label{eqn:weakdual2}
\end{eqnarray}
Now MVSP is nothing but the problem on the right hand side in (\ref{eqn:weakdual2}) 
for the case where $A_i$ is defined as the $m \times n$ matrix 
such that the $(\alpha,\beta)$-th block is $A_{\alpha \beta}$ and other blocks are zero.

The inequality in (\ref{eqn:weakdual1}) is strict in general.
In 2004, Gurvits~\cite{Gurvits04} considered the decision version of 
the Edmonds' problem, and introduced the {\em Edmonds-Rado property} 
for square matrix space ${\cal A}$, which is the property that
${\cal A}$ contains a nonsingular matrix if and only if 
there is no $1$-shrunk subspace 
(or equivalently if there is no vanishing subspace $(X,Y)$ 
with $\dim X + \dim Y > n$). This is the equality case of (\ref{eqn:weakdual1}).
He developed a polynomial time algorithm 
to solve the decision version of 
the Edmonds problem in the case of ${\bf F} = {\bf Q}$.

From 2015, there has been a significant development on this problem; 
the (long) introduction of \cite{GGOW15} is an exciting reading of this development.
Ivanyos, Qiao, and Subrahmanyam~\cite{IQS15a}
noticed that the right hand side of (\ref{eqn:weakdual1}) 
coincides with the {\em non-commutative rank}~\cite{FortinReutenauer04} 
of linear form $\sum_i z_i A_i$ on the {\em skew free field} 
generated by {\em non-commutative} indeterminates $z_i$, 
whereas (\ref{eqn:rank}) is equal to the rank of $\sum_i z_i A_i$ 
over the rational function field ${\bf F}(z_1,z_2,\ldots,z_N)$
for commutative indeterminates $z_i$. 
They showed that the non-commutative rank (nc-rank) of ${\cal A}$ is also given by
\begin{equation}\label{eqn:ncrank}
\max \left\{  \frac{1}{d} \rank \sum_{i=1}^N Z_i \otimes A_i \mid d \geq 1, \ \mbox{$d \times d$ matrix $Z_i$ over ${\bf F}$ for $i=1,2,\ldots,N$}\right\},
\end{equation}
where $\otimes$ is the Kronecker product. 
Namely (\ref{eqn:ncrank}) is 
viewed as the primal problem of the right hand side of (\ref{eqn:weakdual1}) 
in which the strong duality holds.
They also presented the first deterministic algorithm to compute the nc-rank.
Garg, Gurvits, Oliveira, and Wigderson~\cite{GGOW15}
showed in ${\bf F} = {\bf Q}$ that Gurvits' algorithm in \cite{Gurvits04} can be used to compute the nc-rank in polynomial time, 
where a substitution $Z_i$ and 
shrunk subspace $X$ attaining the nc-rank are not obtained in this algorithm.
Derksen and Makam~\cite{DerksenMakam17} 
gave a polynomial bound of $d$ in (\ref{eqn:ncrank}) 
via invariant theoretic arguments. 
By using this result, Ivanyos, Qiao, and Subrahmanyam~\cite{IQS15b}
finally proved that a substitution 
$Z_i$  and shrunk subspace $X$ attaining the nc-rank can be computed in polynomial time.
In particular, by using their algorithm, 
MVSP can be solved in polynomial time for any field ${\bf F}$.
Garg, Gurvits, Oliveira, and Wigderson~\cite{GGOW16} 
used Gurvits' algorithm to solve the feasibility problem on 
the {\em Brascamp-Lieb inequalities} in pseudo polynomial time. 
This problem is essentially WMVSP 
with single row-block ($\mu=1$) and ${\bf F} = {\bf Q}$.

Gurvits' algorithm is analytic, and is based on matrix scaling inspired by quantum information theory.
The algorithm by Ivanyos, Qiao, and Subrahmanyam
is based on the {\em second Wong sequence} for matrix spaces, 
which is a vector space analogue of augmenting path in bipartite matching.
Our algorithm for WMVSP is build on submodularity on modular lattice and 
continuous optimization on CAT(0)-space, 
and is completely different from both algorithms.
Besides the drawback on the bit-length issue, 
its conceptual simplicity and flexibility of our algorithm are notable.
Indeed, our algorithm is easily adapted to compute nc-rank in the same complexity (Remark~\ref{rem:ncrank}).
We believe that the approach presented in this paper will add 
a new perspective to this exciting development on Edmonds' problem.

\paragraph{Organization.}
The rest of this paper is organized as follows.
In Section~\ref{sec:pre}, we summarize 
convex optimization on CAT(0)-space, submodular function, modular lattice, orthoscheme complex, 
and their interplay.
In Section~\ref{sec:MVSP}, 
we first reduce WMVSP to a convex optimization on a CAT(0)-space, 
and apply PPA to prove Theorem~\ref{thm:main'}.
In Section~\ref{sec:DM}, 
we explain implications on block-triangularization of partitioned matrix.

\section{Preliminaries}\label{sec:pre}

\subsection{Convex optimization on CAT(0)-space}

\subsubsection{CAT(0)-space}
Let $S$ be a metric space, and 
let $d: S \times S \to \RR_+$ denote the distance function of $S$.
Let $\diam S := \sup_{x,y \in S} d(x,y)$ 
denote the diameter of $S$.
A {\em path} in $S$ is a continuous map $\gamma$ from $[0,1]$ 
to $S$.
The length of a path $\gamma$ is defined as $\sup \sum_{i=0}^{n-1} d(\gamma(t_i), \gamma(t_{i+1}))$
over $0=t_0 < t_1 < t_2 < \cdots < t_n = 1$ and $n > 0$.
We say that a path $\gamma$ {\em connects} $x,y \in S$ if $\gamma(0) = x$ and $\gamma(1) = y$.
A {\em geodesic} is a path $\gamma$ satisfying 
$d(\gamma(s), \gamma(t)) = d(\gamma(0), \gamma(1)) |s - t|$ 
for every $s,t \in [0,1]$.
Metric space
$S$ is called a {\em geodesic metric space} if
any two points in $S$ is connected by a geodesic, and is said to be {\em uniquely geodesic}
if any two points in $S$ is connected by a unique geodesic.
For points $x,y$ in $S$, let $[x,y]$ 
denote the image of a geodesic $\gamma$ connecting $x,y$
(though a geodesic is not unique).
For $t \in [0,1]$, the point $p$ on $[x,y]$ 
with $d(x,p)/d(x,y) = t$ 
is formally written 
as $(1-t)x + t y$.
A {\em geodesic triangle} of $x,y,z \in S$ is 
the union $[x,y] \cup [y,z] \cup [z,x]$.
In the Euclidean plane $\RR^2$, 
there exist points $\bar x, \bar y, \bar z \in \RR^2$
such that $d(x,y) = \|\bar x - \bar y\|_2$, 
$d(y,z) = \|\bar y - \bar z\|_2$, and
$d(z,x) = \|\bar z - \bar x\|_2$.
%
%A {\em comparison triangle} is 
%the union of 
%$[\bar x, \bar y]$, $[\bar y, \bar z]$, and $[\bar z,\bar x]$.
%
For $p \in [x,y]$, the {\em comparison point} of $p$  is 
the unique point $\bar p$ in $[\bar x, \bar y]$ with $d(x,p) = \| \bar x - \bar p \|_2$.
A geodesic metric space is called {\em CAT$(0)$} if 
for every geodesic triangle $\Delta = [x,y] \cup [y,z] \cup [z,x]$ and every $p,q \in \Delta$, it holds
$
d(p,q) \leq \| \bar p - \bar q\|_{2}. 
$
An intuitive meaning of this definition is that any triangle in $S$ 
is thinner than the corresponding triangle in the Euclidean plane.
See Figure~\ref{fig:cat0}.
	\begin{figure}[t]
		\begin{center}
			\includegraphics[scale=0.6]{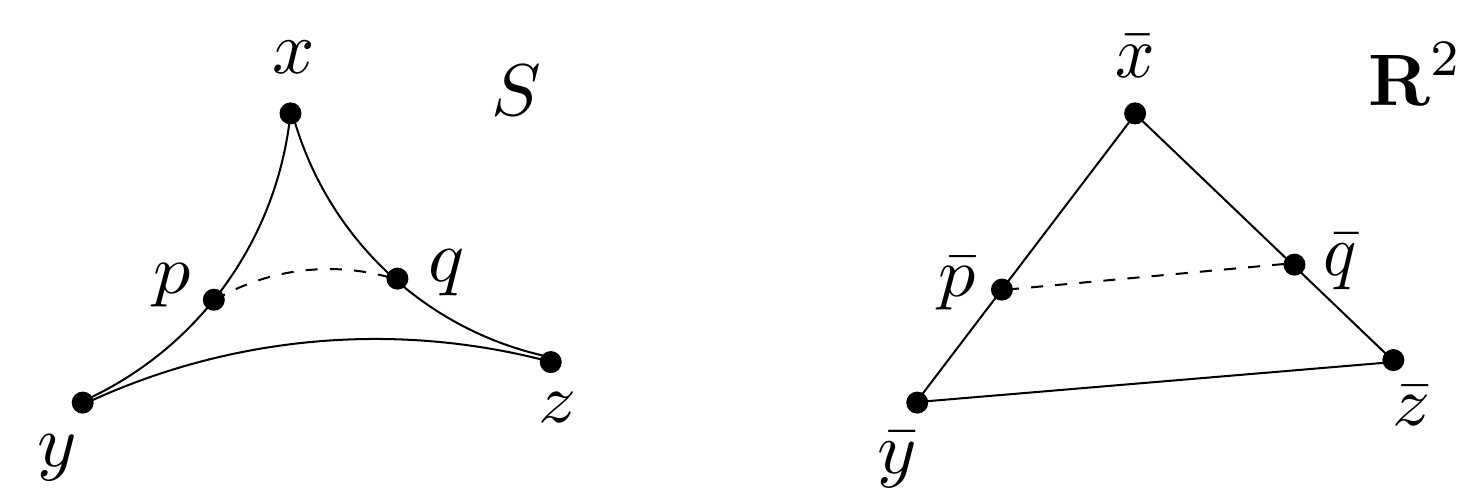}
			\caption{CAT(0) space}
			\label{fig:cat0}
		\end{center}
	\end{figure}\noindent

\begin{Prop}[{\cite[Proposition 1.4]{BrHa}}]\label{prop:uniquely-geodesic}
	A {\rm CAT}$(0)$-space is uniquely geodesic.
\end{Prop}

\subsubsection{Convex function}
Let  $S$ be a {\rm CAT}$(0)$ space. 
A function $f: S\to \RR$ is said to {\em convex}
if
it satisfies 
\begin{equation}\label{eqn:convexity}
 f((1-t) x + t y) \leq (1-t) f(x) + t f(y)
\end{equation}
 for every $x,y \in S$ and $t \in [0,1]$. 
A function $f: S \to \RR$ is said to {\em strongly convex}
with parameter $\kappa > 0$
if
it satisfies 
\begin{equation}\label{eqn:strong_convexity}
f((1-t) x + t y) \leq (1-t) f(x) + t f(y) - \frac{\kappa}{2} t(1-t) d(x,y)^2 
\end{equation}
for every $x,y \in S$ and $t \in [0,1]$. 
A function $f: S \to \RR$ is said to {\em $L$-Lipschitz}
with parameter $L \geq 0$
if
it satisfies 
\begin{equation*}
	|f(x) - f(y)| \leq L d(x,y)
\end{equation*}
for every $x,y \in S$.

\begin{Lem}\label{lem:d^2}
	For any $z \in S$,
	the function $x \mapsto d(z,x)^2$ is strongly convex with $\kappa = 2$, and 
	is $L$-Lipschitz with $L \leq 3\diam S$.
\end{Lem}
The former follows directly from the definition of CAT(0)-space.
The latter follows from $d(z,x)^2 - d(z,y)^2 \leq (d(z,y)+ d(x,y))^2- d(z,y)^2 
= (2d(z,y)+ d(x,y)) d(x,y) \leq (3 \diam S) d(x,y)$.

\subsubsection{Proximal point algorithm}
Let $S$ be a complete CAT(0)-space (which 
is also called an {\em Hadamard space}).
For a convex function $f: S \to \RR$ and $\lambda > 0$
the {\em resolvent} of $f$ is a map $J_\lambda^f: S \to S$ defined by
\begin{equation}
J_{\lambda}^f (x) := \argmin_{y \in S} \left( f(y) + \frac{1}{2\lambda} d(x,y)^2 \right)
\quad (x \in S).
\end{equation}
Since the function $y \mapsto f(y) + \frac{1}{2\lambda} d(x,y)^2$ is strongly convex with parameter $1/\lambda > 0$, 
the minimizer is uniquely determined, and $J_{\lambda}^f$ is well-defined.
The {\em proximal point algorithm (PPA)} is 
to iterate update $x \leftarrow J_{\lambda}^f(x)$.
This simple algorithm generates a sequence converging to a minimizer of $f$
under a mild assumption; see \cite{Bacak13,BacakBook}.
The {\em splitting proximal point algorithm (SPPA)}~\cite{Bacak14,BacakBook}, which we will use, minimizes a convex function $f: S \to \RR$ of the following form
\begin{equation}
	f := \sum_{i=1}^N f_i,
\end{equation}
where each $f_i:S \to \RR$ is a convex function.
Consider a sequence $(\lambda_k)_{k=1,2,\ldots,}$ satisfying
\begin{equation}
\sum_{k=0}^{\infty} \lambda_k = \infty, \quad \sum_{k=0}^\infty \lambda_k^2 < \infty.
\end{equation}
 
\begin{description}
\item[Splitting Proximal Point Algorithm (SPPA)]
\item[$\bullet$] Let $x_0 \in S$ be an initial point.
\item[$\bullet$] For $k=0,1,2,\ldots$, repeat the following:
\[
x_{kN+i} := J_{\lambda_k}^{f_i} (x_{kN+i -1}) \quad (i=1,2,\ldots,N).
\]
\end{description}
Bac\v{a}k~\cite{Bacak14} showed that the sequence generated by
SPPA converges to a minimizer of $f$ if 
$S$ is locally compact.
Ohta and P\'alfia~\cite{OhtaPalfia15} 
proved the sublinear convergence of SPPA if 
$f$ is strongly convex and $S$ is not necessarily locally compact.
\begin{Thm}[\cite{OhtaPalfia15}]\label{thm:OhtaPalfia}
Suppose that 
$f$ is strongly convex with parameter $\epsilon > 0$
and each $f_i$ is $L$-Lipschitz.
Let $x^*$ be the unique minimizer of $f$.
For $0 < a < 1$,
define the sequence $(\lambda_k)$ by
\begin{equation}
\lambda_k := \frac{1- a}{ \epsilon (k+1)}.
\end{equation} 
Then the sequence $(x_\ell)$ generated by SPPA satisfies
\begin{equation}
d(x_{kN}, x^*)^2 \leq \frac{1}{(k+2)^{1-a}} 
   \left(   d(x_0,x^*)^2 + h(a)\frac{L^2N(N+1)}{\epsilon^2} \right)
\quad (k =1,2,\ldots),
\end{equation}
where $h(a) := 2^{2-a}(1 - a)^2 (1+a)/a$. 
\end{Thm}
Note that Ohta and P\'alfia stated this theorem 
assuming $L \geq 1$ but this condition is not used in their proof, and does not affect our argument.

\subsection{Modular lattice}
A {\em lattice} ${\cal L}$ 
is a partially ordered set
such that every pair $p,q$ of elements 
has meet $p \wedge q$ (greatest common lower bound) 
and join $p \vee q$ (lowest common upper bound). 
Let $\preceq$ denote the partial order.
By $p \prec q$ we mean $p \preceq q$ and $p \neq q$.
A pairwise comparable subset of ${\cal L}$, 
arranged as $p_0 \prec p_1 \prec \cdots \prec p_k$,
is called a {\em chain} (from $p_0$ to $p_k$),
where $k$ is called the length. 
In this paper, we only consider lattices in which 
any chain has a finite length.
Let ${\bf 0}$ and ${\bf 1}$ denote the minimum and maximum elements of ${\cal L}$, respectively.
The rank $r(p)$ of element $p$ is defined 
as the maximum length of a chain from ${\bf 0}$ to $p$. 
The rank of lattice ${\cal L}$ is defined as the rank of ${\bf 1}$. 

A lattice ${\cal L}$ is called {\em modular} if 
for every triple $x,a,b$ of elements with $x \preceq b$, 
it holds $x \vee (a \wedge b) = (x \vee a) \wedge b$.
It is known that a modular lattice is exactly such 
a lattice that satisfies
\begin{equation}\label{eqn:modular}
r(p) + r(q) = r(p \wedge q) + r(p \vee q) \quad (p,q \in {\cal L}).
\end{equation}
An element of rank $1$ is called an {\em atom}.
A modular lattice ${\cal L}$ is said to be {\em complemented} 
if every element can be represented as a join of atoms.
A lattice ${\cal L}$ is said to be {\em distributive} 
if $x \vee (y \wedge z)= (x \vee y)\wedge (x \vee z)$ 
and $x \wedge (y \vee z)= (x \wedge y)\vee (x \wedge z)$ 
hold for every triple $x,y,z$ of elements.
A distributive lattice is a modular lattice. 
A complemented distributive lattice is exactly a {\em Boolean lattice}, 
which is a lattice isomorphic to 
the poset $2^{\{1,2,\ldots,n\}}$ of all subsets 
of $\{1,2,\ldots,n\}$ with the inclusion order $\subseteq$.

A function $f:{\cal L} \to \RR$ is said to be {\em submodular} if
\begin{equation}
f(p) + f(q) \geq f(p \wedge q) + f(p \vee q) \quad (p,q \in {\cal L}).
\end{equation}

Let $\check {\cal L}$ denote the opposite lattice of ${\cal L}$, 
where $\check {\cal L}$ and ${\cal L}$ are equal as a set, and
the partial order of $\check {\cal L}$ is the reverse of that of ${\cal L}$.
For a complemented modular lattice ${\cal L}$, 
the opposite lattice $\check {\cal L}$ is also a complemented modular lattice.

A canonical example of a complemented modular lattice is 
the family ${\cal L}$ of all subspaces of a vector space $U$, 
where the partial order is the inclusion order with 
$\wedge = \cap$, and $\vee = +$.
The rank of a subspace $X \in {\cal L}$ is equal to the dimension $\dim X$.
The following equality of dimension is well-known:
\begin{equation}\label{eqn:dim}
\dim X+ \dim X' = \dim (X \cap X') + \dim (X+X') \quad (X,X' \in {\cal L}).
\end{equation}% 

\subsubsection{Basic properties}
Let ${\cal L}$ be a modular lattice of rank $n$, 
and let $r$ be the rank function of ${\cal L}$. 
For $k > 0$,
we denote $p \prec_k q$ if $p \preceq q$ and $r(q) - r(p) = k$.
\begin{Lem}\label{lem:monotone}
	For $p,p',q,q' \in {\cal L}$ with  
	$p \prec_k p'$ and $q \prec_1 q'$, it holds that
	\[
	r(p' \wedge q') - r(p \wedge q') - r(p' \wedge q) + r(p \wedge q) \in \{0,1\}.
	\]
	In particular, the function $u \mapsto r(p' \wedge u) - r(p \wedge u)$ 
	is nondecreasing and takes values from $0$ to $k$.
\end{Lem}
\begin{proof}
	First note that 
	$r(p' \wedge q') - r(p' \wedge q) \in \{0,1\}$ 
	and $r(p \wedge q') - r(p \wedge q) \in \{0,1\}$.
	Indeed, suppose that $r(p' \wedge q') - r(p' \wedge q) > 0$. 
	Then $p' \wedge q' \not \preceq q$ and hence $(p' \wedge q') \vee q = q'$.
	By~(\ref{eqn:modular}), we have $r(p' \wedge q') + r(q) = r(q') + r(p' \wedge q)$, and 
	$r(p' \wedge q') - r(p' \wedge q) = r(q') - r(q) = 1$.
	
	Thus it suffices to consider the case of $p' \wedge q'= p' \wedge q$ 
	(i.e., $r(p' \wedge q') = r(p' \wedge q)$).
	Then $p \wedge p' \wedge q' = p \wedge p' \wedge q$ 
	implies $p \wedge q' = p \wedge q$ (i.e., $r(p' \wedge q') = r(p' \wedge q)$), as required. 
\end{proof}

In the case where ${\cal L}$ is complemented, 
a {\em base} is a set of $n$ atoms $a_1,a_2,\ldots,a_n$ 
with $a_1 \vee a_2 \vee \cdots \vee a_n = {\bf 1}$.
The sublattice $\langle a_1,a_2,\ldots, a_n \rangle$ 
generated by a base $\{a_1,a_2,\ldots,a_n\}$
is called a {\em frame}, which is isomorphic
to a Boolean lattice $2^{\{1,2,\ldots,n\}}$ 
by 
\[
2^{\{1,2,\ldots n\}} \ni X \mapsto \bigvee_{i \in X} a_i \in \langle a_1,a_2,\ldots, a_n \rangle.
\]
\begin{Lem}[{see e.g.,\cite{Gratzer}}]\label{lem:frame}
Let ${\cal C}$ and ${\cal D}$ be (maximal) chains in $\cal L$.
The sublattice generated by ${\cal C}$ and ${\cal D}$ is distributive.
If ${\cal L}$ is complemented, 
then there is a frame $\langle a_1,a_2,\ldots, a_n \rangle \subseteq {\cal L
}$ containing $\cal C$ and $\cal D$.
\end{Lem}
A complemented modular lattice is viewed as a {\em spherical building of type A}~\cite{BuildingBook} 
The latter property of this lemma features the axiom of building, and 
is particularly important for us; we provide a proof based on~\cite[Section 4.3]{BuildingBook}. 
	\begin{figure}[t]
		\begin{center}
			\includegraphics[scale=0.6]{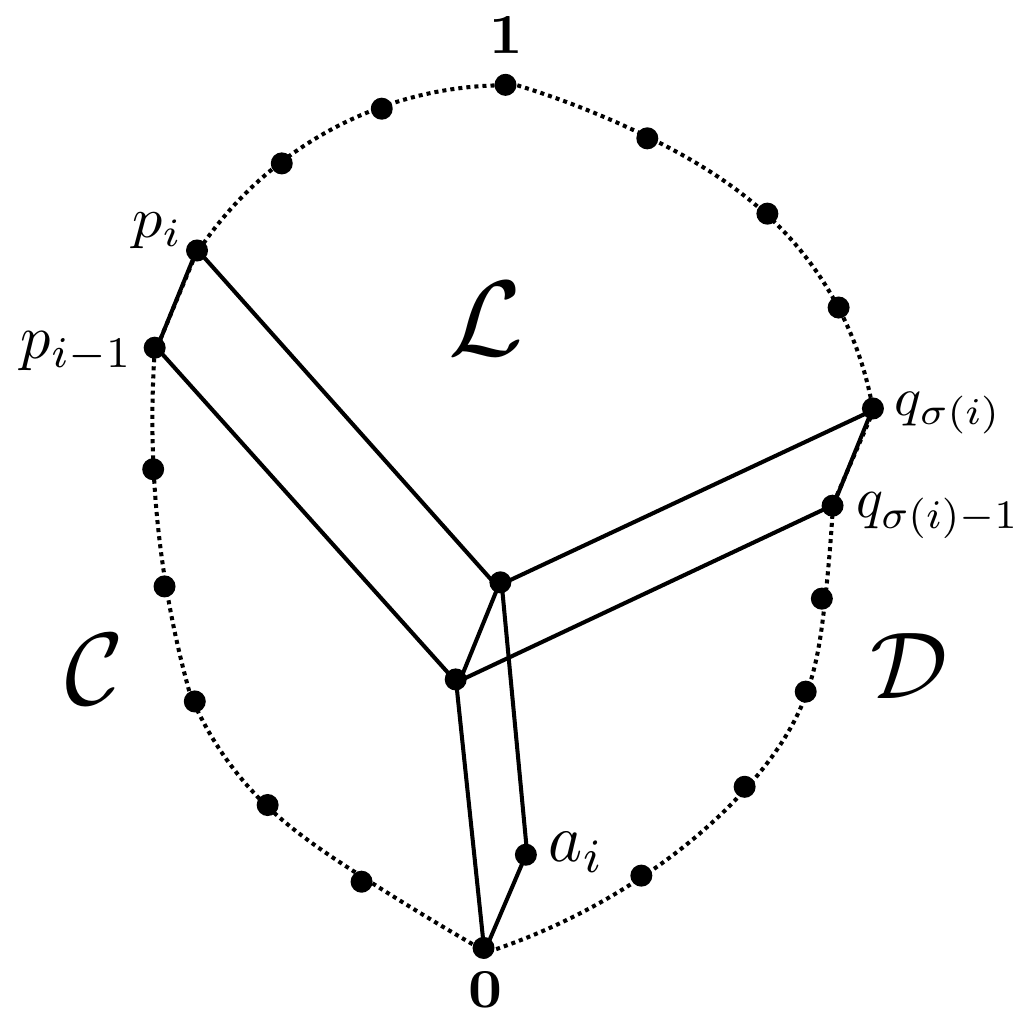}
			\caption{Construction of a frame including two chains ${\cal C}$ and ${\cal D}$}
			\label{fig:apartment}
		\end{center}
	\end{figure}\noindent
\begin{proof}
	Suppose that ${\cal C} = ({\bf 0} = p_0 \prec p_1 \prec \cdots \prec p_n = {\bf 1})$
	and ${\cal D} = ({\bf 0} = q_0 \prec q_1 \prec \cdots \prec q_n = {\bf 1})$.
	We first show 
	\begin{Clm}
		There exists a bijection $\sigma$ on $\{1,2,\ldots,n\}$ such that
		$p_{i-1} \wedge q_{\sigma(i)-1} \prec_1 p_i \wedge q_{\sigma(i)}$ for each $i$.
	\end{Clm}
	Assume the claim.
	By complementarity, for each $i$, 
	we can choose an atom $a_i$ such that 
	$(p_{i-1} \wedge q_{\sigma(i)-1}) \vee a_i = p_{i} \wedge q_{\sigma(i)}$.
	Then it holds
	$p_{i-1} \vee a_i = p_i$ and $q_{\sigma(i)-1} \vee a_i = q_{\sigma(i)}$.
	Consequently, all $p_i$ and $q_i$ are represented as joins of $a_1,a_2,\ldots,a_n$. See Figure~\ref{fig:apartment} for intuition.
	
	We prove the claim.
	By Lemma~\ref{lem:monotone}, 
	for each $i \in \{1,2,\ldots,n\}$ there uniquely exists $j \in \{1,2,\ldots,n\}$
	such that
	\begin{equation}\label{eqn:unique1}
	r(p_i \wedge q_j) - r(p_{i-1} \wedge q_j) = 1,\ r(p_{i} \wedge q_{j-1}) - r(p_{i-1} \wedge q_{j-1}) = 0. 
	\end{equation}
	In particular, it holds that
	\begin{equation}\label{eqn:preceq}
	p_i \wedge q_{j-1} = p_{i-1} \wedge q_{j-1} \preceq p_{i-1} \wedge q_j \prec_1 p_i \wedge q_j.
	\end{equation}
	Here $p_i \wedge q_j \not \preceq q_{j-1}$ must hold; 
	if $p_i \wedge q_j \preceq q_{j-1}$, then 
	$p_i \wedge q_j = p_i \wedge p_i \wedge q_j \preceq p_i \wedge q_{j-1}$, which contradicts (\ref{eqn:preceq}).
	Thus $(p_i \wedge q_j) \vee q_{j-1} = q_j$, implying 
	$p_i \wedge q_{j-1} \prec_1 p_i \wedge q_j$ (by (\ref{eqn:modular})). 
	By (\ref{eqn:preceq}), 
	it necessarily holds that 
	$p_{i-1} \wedge q_{j-1} = p_{i-1} \wedge q_j  \prec_1 p_i \wedge q_{j}$.

    Thus we can define the map $\sigma$
	by associating $i$ with $\sigma(i) := j$ with property (\ref{eqn:unique1}).
	This map is injective, and hence bijective. Indeed, 
	by (\ref{eqn:unique1}), we have 
	$r(p_i \wedge q_j) - r(p_{i} \wedge q_{j-1}) - r(p_{i-1} \wedge q_j) + r(p_{i-1} \wedge q_{j-1}) = 1$, and
	\begin{equation*}
	r(p_i \wedge q_j) - r(p_{i} \wedge q_{j-1}) = 1,\  
	r(p_{i-1} \wedge q_{j}) - r(p_{i-1} \wedge q_{j-1}) = 0.
	\end{equation*}
	This means that interchanging the roles of $i,j$ yields the inverse map of $\sigma$.
\end{proof} 
Suppose that ${\cal L}$ is the lattice of all vector subspaces of a vector space, 
and that we are given two chains ${\cal C}$ and ${\cal D}$ of vector subspaces, 
where each subspace $X$ in the chains is given by a matrix $A$ with ${\rm Im} A = X$ 
or/and a matrix $B$ with $\ker B = X$.
The above proof can be implemented via rank computation/Gaussian elimination, 
and obtain vectors  $a_1,a_2,\ldots,a_n$ 
with ${\cal C}, {\cal D} \subseteq \langle a_1,a_2,\ldots, a_n \rangle$
in polynomial time.

Let $U$ and $V$ be vector spaces of dimension $m$ and $n$, respectively, 
and let $A: U \times V \to {\bf F}$ be a bilinear form.
Let ${\cal L}$ and ${\cal M}$ be the lattices of 
all vector subspaces of $U$ and of $V$, respectively.
Consider the opposite $\check{\cal M}$.
Define $R = R^A: {\cal L} \times \check{\cal M} \to \ZZ$ by
\begin{equation}\label{eqn:R}
R(X,Y) := \rank A|_{X \times Y} \quad ((X,Y) \in {\cal L} \times \check{\cal M}),
\end{equation}
where $A|_{X \times Y}: X \times Y \to {\bf F}$ is 
the restriction of $A$ to $X \times Y$, and $\rank$ is 
the rank of the matrix representation.
Then $R$ is submodular; 
an equivalent statement is in \cite[Lemma 4.2]{IwataMurota95}. 
\begin{Lem}\label{lem:R_is_submo}
	For $(X,Y),(X',Y') \in {\cal L} \times {\cal M}$, it holds
	\begin{equation}\label{eqn:submo_R}
	R(X,Y) + R(X',Y') \geq R(X \cap X', Y + Y') + R(X + X', Y \cap Y').
	\end{equation}
	Thus $R$
	is a submodular function on ${\cal L} \times \check{\cal M}$.
\end{Lem}
\begin{proof}
	By Lemma~\ref{lem:frame}, 
	there is a base $\{ a_1,a_2,\ldots,a_m\}$ of ${\cal L}$ 
	with $X, X', X \cap X', X + X' \subseteq \langle a_1,a_2,\ldots,a_m \rangle$,  
	and there is  
	a base $\{ b_1,b_2,\ldots,b_n\}$ of ${\cal M}$ 
	with $Y, Y', Y \cap Y', Y + Y' \subseteq \langle b_1,b_2,\ldots,b_n \rangle$.
	Consider the matrix representation $A = (a_{ij})$ 
	with respect to these bases, i.e., $a_{ij} := A(a_i,b_j)$.
	For $I \subseteq \{1,2,\ldots, m\}$ and $J \subseteq \{1,2,\ldots,n\}$, 
	let $A[I,J] := (a_{ij}: i \in I,j \in J)$ 
	be the submatrix of $A$ with row set $I$ and column set $J$.
	Then (\ref{eqn:submo_R}) follows from 
	the  well-known rank inequality
	\begin{equation*}
	\rank A[I,J] + \rank A[I',J'] \geq \rank A[I \cap I',J \cup J'] 
	   + \rank A[I \cup I',J \cap J']
	\end{equation*}
    for $I,I' \subseteq \{1,2,\ldots,m\}$ and $J,J' \subseteq \{1,2,\ldots,n\}$; 
    see~\cite[Proposition 2.1.9]{MurotaBook}.
\end{proof}

\subsubsection{Orthoscheme complex} 
The $n$-dimensional {\em orthoscheme} is the simplex in $\RR^n$ with vertices
\[
0, e_1, e_1 + e_2, e_1 + e_2 + e_3, \ldots, e_1 + e_2 + \cdots + e_n,
\]
where $e_i$ is the $i$th unite vector; 
see Figure~\ref{fig:ortho} for the $3$-dimensional orthoscheme.
	\begin{figure}[t]
		\begin{center}
			\includegraphics[scale=0.7]{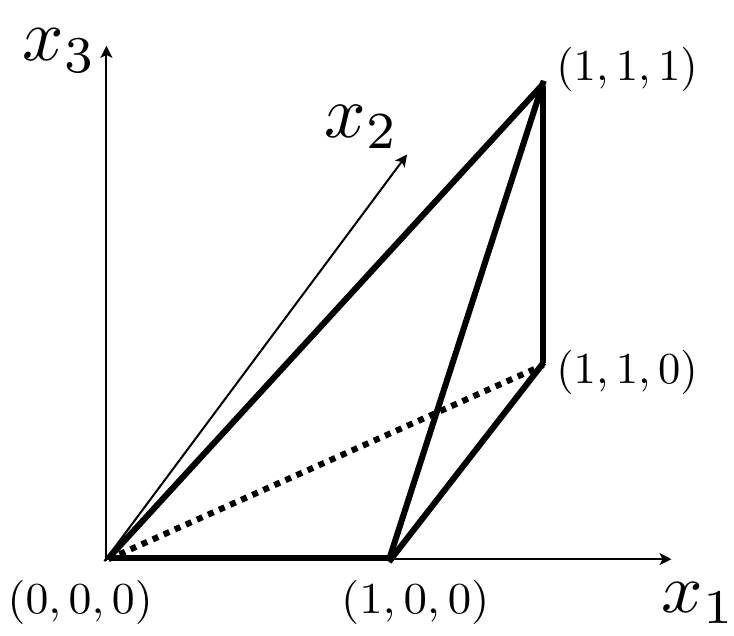}
			\caption{Orthoscheme}
			\label{fig:ortho}
		\end{center}
	\end{figure}\noindent
An orthoscheme complex, introduced by Brady and McCammond~\cite{BM10} in the context of geometric group theory,  
is a metric simplicial complex obtained by gluing orthoschemes.
Let ${\cal L}$ be a modular lattice of rank $n$.
Let $F({\cal L})$ be the free $\RR$-module over ${\cal L}$, 
i.e., the set of formal (finite) linear combinations
$x = \sum_{p \in {\cal L}} \lambda(p) p$ 
such that each coefficient $\lambda(p)$ is in  $\RR$
and the set of elements $p$ with nonzero coefficient, which we call  
the {\em support} of $x$, 
is finite.
Let $K({\cal L})$ 
be the subset of elements 
$x = \sum_{p \in {\cal L}} \lambda(p) p \in F({\cal L})$
such that $\lambda(p) \geq 0$ for $p \in {\cal L}$, $\sum_{p \in {\cal L}} \lambda(p) = 1$, and the support of $x$ 
is a chain of ${\cal L}$.
Namely $K({\cal L})$ is the geometric realization of the order complex of ${\cal L}$.
The subset of $K({\cal L})$ consisting of formal combinations of some chain ${\cal  C}$
is called a {\em simplex} of $K({\cal L})$.
For a maximal simplex $\sigma$ corresponding to 
a maximal chain ${\cal C} = p_0 \prec p_1 \prec \cdots \prec p_n$, 
define a map $\varphi_{\sigma}$ from $\sigma$ to 
the $n$-dimensional orthoscheme
by
\begin{equation*}\label{eqn:phi_sigma}
\varphi_{\sigma}(x) = \sum_{i=1}^{n} \lambda_i (e_1 + e_2 + \cdots + e_{i}) 
\quad (x = \sum_{i=0}^n \lambda_i p_i \in \sigma).
\end{equation*}
Then a metric $d_{\sigma}$ on each simplex $\sigma$ is defined by
\begin{equation}\label{eqn:dsigma}
d_{\sigma}(x,y) := \| \varphi_\sigma (x) - \varphi_\sigma (y) \|_2 \quad (x,y \in \sigma). 
\end{equation}
The length of a path $\gamma: [0,1] \to K({\cal L})$
is defined as $\sup \sum_{i=0}^{m-1} d_{\sigma_i}(\gamma(t_i),\gamma(t_{i+1}))$, 
where the $\sup$ is taken over all
$0 = t_0 < t_1 < t_2 < \cdots < t_m = 1$ $(m \geq 1)$ 
such that $\gamma([t_i,t_{i+1}])$ belongs to a simplex $\sigma_i$ for each $i$.
The metric on $K({\cal L})$ is (well-)defined as above.
The resulting metric space $K({\cal L})$ 
is called the {\em orthoscheme complex} of ${\cal L}$.
Then $K({\cal L})$ is a complete 
geodesic metric space (by Bridson's theorem~\cite[Theorem 7.19]{BrHa}).
\begin{Thm}[\cite{CCHO,HKS17}]
Let ${\cal L}$ be a modular lattice of rank $n$.
The orthoscheme complex $K({\cal L})$ is 
a complete CAT(0)-space.
\end{Thm}

\begin{Lem}[\cite{BM10,CCHO}]\label{lem:product}
	Let ${\cal L}$ and ${\cal M}$ be modular lattices.
	Define a metric $d$ on $K({\cal L}) \times K({\cal M})$ by
	\[
	d((x,y),(x',y')) := \sqrt{ d(x,x')^2  + d(y,y')^2} \quad ((x,y),(x',y') \in K({\cal L}) \times K({\cal M})).
	\]
	Then $K({\cal L}) \times K({\cal M})$ is isometric to $K({\cal L} \times {\cal M})$, where
%	where
%	the isometry is given by
%	\begin{eqnarray}
%	&& z = \sum_{k=0}^{n+m} \kappa_k (p_k,q_k) \mapsto \left(\sum_{i} \lambda_i \tilde p_{i}, \sum_j \mu_j \tilde q_{j} \right), \nonumber \\
%	&& \lambda_i := \sum_{k: p_k = p_i} \kappa_k,\quad \mu_j := \sum_{k: q_k = q_j} \kappa_k. \label{eqn:map}
%	\end{eqnarray} 
the isometry $\phi: K({\cal L}) \times K({\cal M}) \to K({\cal L} \times {\cal M})$ is given by the following algorithm:
\begin{description}
	\item[Input:] $(x,y) \in K({\cal L}) \times K({\cal M})$.
	\item[Output:] 
	$z = \phi(x,y) \in K({\cal L} \times {\cal M})$.
	\item[0:] Let $z := 0$
	\item[1:] If $(x,y) = (0,0)$, then return $z$.
	\item[2:] Choose the maximum element $p$ from the support of $x$ and the maximum element $q$ from the support of $y$. 
	\item[3:] Let $\lambda$ be the minimum of the coefficient of $p$ in $x$ and that  
	of $q$ in~$y$. 
	Let $x \leftarrow x - \lambda p$, $y \leftarrow y - \lambda q$, 
	and $z \leftarrow z + \lambda (p,q)$. Go to {\bf 1}.
\end{description}
\end{Lem}
The orthoscheme complex of a Boolean lattice is a Euclidean cube as follows, 
where $1_{X} \in \{0,1\}^n$ is the characteristic vector 
of $X \subseteq \{1,2,\ldots,n\}$ defined by 
$(1_{X})_i = 1 \Leftrightarrow i \in X$. 
\begin{Lem}[\cite{BM10,CCHO}]\label{lem:Boolean}
Let ${\cal L}$ be a Boolean lattice $2^{\{1,2,\ldots,n\}}$.
The orthoscheme complex $K({\cal L})$ is isometric 
to the $n$-cube $[0,1]^n$ in $\RR^n$, where an isometry is given by
\begin{equation}\label{eqn:isometry}
x = \sum_{i} \lambda_i X_i
\mapsto  \sum_{i} \lambda_i 1_{X_i}.
\end{equation}
\end{Lem}

\begin{Lem}[\cite{CCHO}]
	Let ${\cal L}$ be a complemented modular lattice of rank $n$, and 
	let ${\cal F}$ be a frame of ${\cal L}$.
	Then $K({\cal F}) \simeq [0,1]^n$ is 
	an isometric subcomplex of $K({\cal L})$.
\end{Lem}
\begin{Cor}\label{cor:diam}
	Let ${\cal L}$ be a complemented modular lattice of rank $n$.
	Then $\diam K({\cal L}) = \sqrt{n}$.
\end{Cor}
\begin{proof}
	For two points $x,y \in K({\cal L})$, there is a frame ${\cal F}$ 
	such that $x,y \in K({\cal F})$ (by Lemma~\ref{lem:frame}).
	Since $K({\cal F}) \simeq [0,1]^n$ and $K({\cal F})$ is an isometric subspace, 
	the distance $d(x,y)$ 
	is bounded by the diameter $\sqrt{n}$ of $[0,1]^n$, 
	which is attained by $x = {\bf 0}$ and $y = {\bf 1}$.
\end{proof}
A frame ${\cal F} = \langle a_1,a_2,\ldots,a_n \rangle$ 
is isomorphic to Boolean lattice $2^{\{1,2,\ldots,n\}}$ 
by $a_{i_1} \vee a_{i_2} \vee \cdots \vee a_{i_k} \mapsto \{i_1,i_2,\ldots,i_k\}$. 
Also the subcomplex $K({\cal F})$ is viewed as an $n$-cube $[0,1]^n$, and
a point $x$ in $K({\cal F})$ 
is viewed as $x = (x_1,x_2,\ldots,x_n) \in [0,1]^n$
via isometry $(\ref{eqn:isometry})$.
This $n$-dimensional vector $(x_1,x_2,\ldots,x_n)$ is called 
the {\em ${\cal F}$-coordinate} of $x$. 
From ${\cal F}$-coordinate $(x_1,x_2,\ldots,x_n)$, 
the original expression of $x$ is recovered by
sorting $x_1,x_2,\ldots,x_n$ in decreasing order as: $x_{i_1} \geq x_{i_2} \geq \cdots \geq x_{i_n}$, and letting
\begin{equation}\label{eqn:recover}
x = (1- x_{i_1}){\bf 0} + \sum_{k=1}^n (x_{i_k} - x_{i_{k+1}}) (a_{i_1} \vee a_{i_2} \vee \cdots \vee a_{i_k}),
\end{equation}
where $x_{i_{n+1}} := 0$.

\subsubsection{Lov\'asz extension}
We here introduce the Lov\'asz extension 
for a function on a modular lattice ${\cal L}$.
For a function $f:{\cal L} \to \RR$, 
the {\em Lov\'asz extension} $\overline f: K({\cal L}) \to \RR$ of $f$ is defined by
\begin{equation*}
\overline f (x) := \sum_{i} \lambda_i f(p_i) \quad  (x = \sum_{i} \lambda_i p_i \in K({\cal L}) ).
\end{equation*}
In the case where ${\cal L} = 2^{\{1,2,\ldots,n\}}$, 
this definition of the Lov\'asz extension coincides with the original one~\cite{FujiBook,Lovasz83}
by $K({\cal L}) \simeq [0,1]^n$ (Lemma~\ref{lem:Boolean}).
\begin{Thm}[\cite{HH16L-convex}]\label{thm:Lovasz_ext}
Let ${\cal L}$ be a modular lattice.
For a function $f: {\cal L} \to \RR$, the following conditions are equivalent:
\begin{itemize}
\item[{\rm (1)}] $f$ is submodular. 
\item[{\rm (2)}] $\overline{f}$ is convex
\end{itemize}
\end{Thm}
\begin{proof}[Sketch of proof] 
	For two points $x,y \in K({\cal L})$, 
	there is a frame ${\cal F}$ such that $K({\cal F})$ contains $x,y$.
    Also $K({\cal F})$ is an isometric subspace of $K({\cal L})$.
	Therefore the geodesic $[x,y]$ belongs to $K({\cal F})$.
	Hence, a function on 
	$K({\cal L})$ is convex if and only 
	if it is convex on $K({\cal F})$ for every frame ${\cal F}$.
	For any frame ${\cal F}$, 
	the restriction of a submodular function $f:{\cal L} \to \RR$ to ${\cal F}$
	is a usual submodular function on Boolean lattice 
	${\cal F} \simeq 2^{\{1,2,\ldots,n\}}$.
	Hence
	$\overline f:K({\cal F}) \to \RR$ is viewed as
	the usual Lov\'asz extension by $[0,1]^n \simeq K({\cal F})$, and is convex.
\end{proof}

The rank function $r$ is submodular. The Lov\'asz extension $\overline{r}$ of $r$ 
is written by 
the $l_1$-metric on $K({\cal L})$.
Here the $l_1$-metric $d_1$ is obtained by
replacing $\| \cdot \|_2$ by $\| \cdot \|_1$ in (\ref{eqn:dsigma}), i.e.,
\begin{equation*}\label{eqn:d1_sigma}
d_{\sigma}(x,y) := \| \varphi_\sigma (x) - \varphi_\sigma (y) \|_1 \quad (x,y \in \sigma). 
\end{equation*}
The $l_1$-metric on $K({\cal L})$ is denoted by $d_1$. 
%the subcomplex $K({\cal F})$ for a frame ${\cal F}$ is also isometric 
%in $K({\cal L})$ with respect to the $l_1$-metric~\cite{CCHO}. 
The function $x \mapsto d_1({\bf 0}, x)$ is simply written as $d_1$.
\begin{Lem}\label{lem:Lovasz_r}
%	Let ${\cal L}$ be a modular lattice with rank function $r$.
	The Lov\'asz extension $\overline r$ of the rank function $r$
	is equal to $d_1$.
\end{Lem}
\begin{proof}
For $x = \sum_{i=0}^n \lambda_i p_i \in K({\cal L})$,
consider the simplex $\sigma$ formed by $p_i$'s.
Then we have 
\begin{equation*}
d_1({\bf 0},x) = \| \varphi_\sigma(x) \|_1 
=  \| \sum_{i=1}^n \lambda_i (e_1+e_2 + \cdots + e_i) \|_1 = \sum_{i=0}^n \lambda_i i = \sum_{i=0}^n \lambda_i r(p_i) = \overline{r}(x).
\end{equation*}
\end{proof}

The following lemma will be used to obtain a minimizer of a function on ${\cal L}$ 
from an approximate minimizer of its Lov\'asz extension.
\begin{Lem}~\label{lem:minimizer}
%Let ${\cal L}$ be a modular lattice, and 
Let $f:{\cal L} \to \ZZ$ be an integer-valued function, and 
let $p^* \in {\cal L}$ be a minimizer of $f$.
For $x \in K({\cal L})$,
if $\overline{f}(x) - f(p^*) < 1$, then
there exists a minimizer of~$f$ in the support of $x$.
\end{Lem}
\begin{proof}
	Suppose that $x = \sum_{i} \lambda_i p_i$.
Suppose to the contrary that all $p_i$'s satisfy $f(p_i) > f(p^*)$.
Then $f(p_i) \geq f(p^*) + 1$.
Hence $\overline f(x) = \sum_{i} \lambda_i f(p_i) 
\geq \sum_{i}\lambda_i (f(p^*)+1) = f(p^*) + 1$.
However this contradicts $\overline{f}(x) - f(p^*) < 1$.
\end{proof}

The following lemma will be used to estimate the Lipschitz constant
of the Lov\'asz extension.
\begin{Lem}\label{lem:Lipschitz}
The Lov\'asz extension $\overline{f}$ 
of $f: {\cal L} \to \RR$ 
is $L$-Lipschitz with
\begin{equation*}
L \leq 2 \sqrt{n} \max_{p \in {\cal L}} |f(p)|.
\end{equation*}
\end{Lem}
\begin{proof}
We first show that the restriction $\overline{f}|_{\sigma}$ of $\overline{f}$ to any maximal simplex $\sigma$ is $L$-Lipschitz with 
$
L \leq 2 \sqrt{n} \max_{p \in {\cal L}} |f(p)|
$.
Suppose that $\sigma$ corresponds to 
a chain ${\bf 0} = p_0 < p_1 < \cdots < p_n = {\bf 1}$.
Let $x = \sum_{k}\lambda_k p_k$ and $y = \sum_{k}\mu_k p_k$ be points in $\sigma$. For $k=0,1,2,\ldots,n$,
define $u_k$ and $v_k$ by
\[
u_k := \lambda_{k} + \lambda_{k+1} + \cdots + \lambda_{n}, \quad 
v_k := \mu_{k} + \mu_{k+1} + \cdots + \mu_{n}.
\]
%\begin{eqnarray*}
%u_k &:=& \lambda_{k} + \lambda_{k+1} + \cdots + \lambda_{n}, \\
%v_k &:=& \mu_{k} + \mu_{k+1} + \cdots + \mu_{n}.
%\end{eqnarray*}
Then $d_{\sigma}(x,y)$ is given by
\begin{equation*}
d_{\sigma}(x,y) = \sqrt{\sum_{k=1}^n (u_k - v_k)^2}. 
\end{equation*}
Letting $C := \max_{p \in {\cal L}} |f(p)|$, 
we have
\begin{eqnarray*}
&& |\overline{f}(x) - \overline{f}(y)| = \left| 
\sum_{k=0}^n (\lambda_k - \mu_k) f(p_k) \right| 
\leq C\sum_{k=0}^n |\lambda_k - \mu_k| \\ && 
= C\sum_{k=0}^n | u_{k}- u_{k+1} - (v_{k} - v_{k+1})|  \leq 
2 C \sum_{k=1}^n |u_k - v_k|   \leq 2 \sqrt{n} C  \sqrt{\sum_{k=1}^n (u_k - v_k)^2},
\end{eqnarray*}
where we let $u_0 = v_0 := 1$ and $u_{n+1}= v_{n+1} := 0$.
Thus $\overline{f}|_{\sigma}$ is $2 \sqrt{n} C$-Lipschitz. 

Next we show that $\overline{f}$ is  $2 \sqrt{n} C$-Lipschitz. 
For any $x,y\in K({\cal L})$, 
choose the geodesic $\gamma$ between $x$ and $y$, 
and $0=t_0 < t_1 < \cdots < t_m = 1$ 
such that $\gamma([t_i,t_{i+1}])$ belongs  to simplex $\sigma_i$.
Then we have
\[
|\overline f(x) -   \overline f(y)|  \leq  
\sum_{i=1}^m | \overline{f}(\gamma(t_i)) - \overline f (\gamma(t_{i-1})) |  \leq  2 \sqrt{n} C \sum_{i=1}^{m} d_{\sigma_i}(\gamma(t_i),\gamma(t_{i-1})) =  2 \sqrt{n} C d(x,y).
\]
\end{proof}

\section{Maximum vanishing subspace problem}\label{sec:MVSP}

%We deal with a weighted generalization of MVSP, 
%the {\em weighted maximum vanishing subspace problem (WMVSP)}.
%We are given a partitioned matrix 
%$A = (A_{\alpha \beta})$ of type $(m_{1},m_{2},\ldots,m_{\mu}; n_{1}, n_2,\ldots,n_{\nu})$, 
%and nonnegative integer weights $C_{\alpha},D_{\beta}$ 
%for $1 \leq \alpha \leq \mu$ and $1 \leq \beta \leq \nu$.
%WMVSP asks to maximize 
%\begin{equation*}
%\sum_{\alpha}C_{\alpha} \dim X_{\alpha} + \sum_{\beta} D_{\beta} \dim Y_{\beta} 
%\end{equation*}
%over all vanishing subspaces 
%$(X_{1},X_{2},\ldots,X_{\mu},Y_{1},Y_2,\ldots,Y_{\nu})$.
%Let $m := \sum_{\alpha} m_{\alpha}$ and $n := \sum_{\beta} n_{\beta}$, i.e.,  $A = (A_{\alpha \beta})$ is an $m \times n$ matrix. 
%The goal of this section is 
%to prove the pseudo-polynomial time solvability of WMVSP:
%\begin{Thm}\label{thm:main'}
%WMVSP can be solved in time polynomial in
%$m,n$ and $W$, where $W$ is the maximum of weights $C_{\alpha}, D_{\beta}$.
%\end{Thm}
%
\subsection{CAT(0)-space relaxation}
Suppose that we are given an instance of WMVSP:
a partitioned matrix 
$A = (A_{\alpha \beta})$ of type $(m_{1},m_{2},\ldots,m_{\mu}; n_{1}, n_2,\ldots,n_{\nu})$
and nonnegative integer weights $C_{\alpha},D_{\beta}$ 
for $1 \leq \alpha \leq \mu$ and $1 \leq \beta \leq \nu$.
Let $m = \sum_{\alpha} m_\alpha$ and $n = \sum_{\beta} n_\beta$.
First we formulate WMVSP as 
an unconstrained submodular function minimization 
over a complemented modular lattice.
Let ${\cal L}_{\alpha}$ and ${\cal M}_{\beta}$ 
denote the lattices of all 
vector subspaces of ${\bf F}^{m_\alpha}$ and of ${\bf F}^{n_\beta}$, respectively.
Let $R_{\alpha \beta} := R^{A_{\alpha, \beta}}$; see (\ref{eqn:R}) for the definition of $R$.
Then the condition (\ref{eqn:vanishing}) is written as
\begin{equation}\label{eqn:R=0}
R_{\alpha \beta}(X_{\alpha},Y_{\beta}) = 0 \quad (1 \leq \alpha \leq \mu, 1 \leq \beta \leq \nu).
\end{equation}
By using $R_{\alpha \beta}$ as penalty terms, 
WMVSP is equivalent to the following unconstrained problem:
\begin{eqnarray*}
{\rm WMVSP}_{R} : && \\
{\rm Min.} &&  
- \sum_{\alpha}C_{\alpha} \dim X_{\alpha} - \sum_{\beta} D_{\beta} \dim Y_{\beta} 
+ M \sum_{\alpha, \beta} R_{\alpha \beta}(X_{\alpha},Y_{\beta}) \\
{\rm s.t.} && (X_1,X_2,\ldots,X_{\mu}, Y_1,Y_2,\ldots,Y_{\nu}) \in 
\prod_{\alpha}{\cal L}_\alpha \times \prod_{\beta} \check{\cal M}_\beta, 
\end{eqnarray*}
where the penalty parameter $M > 0$ is chosen as 
\[
M := \sum_{\alpha}C_{\alpha} m_{\alpha} + \sum_{\beta} D_{\beta} n_{\beta}+1.
\]
\begin{Lem}
Any optimal solution of WMVSP$_R$ is optimal to WMVSP	
\end{Lem}
\begin{proof}
	It suffices to show that any optimal 
	solution of WMVSP$_R$ satisfies the condition~(\ref{eqn:R=0}). 
	Indeed, if $R_{\alpha \beta}(X_{\alpha},Y_{\beta}) > 0$ for 
	$(X_1,\ldots,X_{\mu},Y_1,\ldots,Y_{\nu})$ and some $\alpha, \beta$
	then the objective value of WMVSP$_R$ is positive, 
	and $(X_1,\ldots,X_{\mu},Y_1,\ldots,Y_{\nu})$ is never optimal 
	(since the trivial solution 
	$(\{0\},\ldots,\{0\},\{0\},\ldots,\{0\})$ has the objective value zero).
\end{proof}
By (\ref{eqn:dim}), Lemmas~\ref{lem:R_is_submo} and \ref{lem:Lovasz_r}, we have:
\begin{Lem}
	The objective function of WMVSP$_R$ is submodular on $\prod_{\alpha}{\cal L}_\alpha \times \prod_{\beta} \check{\cal M}_\beta$, where the Lov\'asz extension is given
	by
	\begin{eqnarray*}
	&& (x_1,x_2,\ldots,x_\mu,y_1,y_2,\ldots,y_{\nu})  \\
	&& \quad \mapsto - \sum_{\alpha}C_{\alpha} d_1(x_{\alpha}) - \sum_{\beta} D_{\beta} d_1(y_{\beta}) 
	+ M \sum_{\alpha, \beta} \overline{R_{\alpha \beta}}(x_{\alpha},y_{\beta}).
	\end{eqnarray*}
\end{Lem}
Recall that $d_1$ is the function $x \mapsto d_1({\bf 0}, x)$.
In particular, WMVSP$_R$ is equivalent to the following continuous optimization on CAT(0) space:
\begin{eqnarray*}
	\overline{{\rm WMVSP}}_R: && \\
	{\rm Min.} &&  
	- \sum_{\alpha}C_{\alpha} d_1 (x_{\alpha}) - \sum_{\beta} D_{\beta} d_1(y_{\beta}) 
	+ M \sum_{\alpha, \beta} \overline{R_{\alpha \beta}}(x_{\alpha},y_{\beta}) \\
	{\rm s.t.} && (x_1,x_2,\ldots,x_{\mu}, y_1,y_2,\ldots,y_{\nu}) \in 
	\prod_{\alpha}K({\cal L}_\alpha) \times \prod_{\beta} K(\check{\cal M}_\beta), 
\end{eqnarray*}
where $K(\prod_{\alpha}{\cal L}_\alpha \times \prod_{\beta} \check{\cal M}_\beta)$ is considered as $\prod_{\alpha}K({\cal L}_\alpha) \times \prod_{\beta} K(\check{\cal M}_\beta)$ 
by Lemma~\ref{lem:product}.
By Theorem~\ref{thm:Lovasz_ext}, $\overline{{\rm WMVSP}}_R$ is a convex optimization problem. 
\begin{Lem}\label{lem:Lipschitz_g0}
	The objective function of $\overline{{\rm WMVSP}}_R$ is convex.
\end{Lem}

\subsection{Proximal point algorithm for MVSP}
We are going to apply SPPA
to the following perturbed problem of $\overline{\rm WMVSP}_{R}$:
\begin{eqnarray*}
	\overline{\rm WMVSP}^+_{R}: && \\
	{\rm Min.} &&  
	- \sum_{\alpha}C_{\alpha} d_1 (x_{\alpha}) - \sum_{\beta} D_{\beta} d_1(y_{\beta}) 
	+ M \sum_{\alpha, \beta} \overline{R_{\alpha \beta}}(x_{\alpha},y_{\beta}) \\
	&& \quad \quad + \epsilon \left(\sum_{\alpha} d^2(x_{\alpha}) + \sum_{\beta} d^2( y_{\beta}) \right)\\
	{\rm s.t.} && (x_1,x_2,\ldots,x_{\mu}, y_1,y_2,\ldots,y_{\nu}) \in 
	\prod_{\alpha}K({\cal L}_\alpha) \times \prod_{\beta} K(\check{\cal M}_\beta), 
\end{eqnarray*}
where the function $x \mapsto d({\bf 0}, x)^2$ 
is denoted by $d^2$,
and the parameter $\epsilon > 0$ is chosen as
\[
\epsilon := \frac{1}{4(n+m)}.
\]
The main reason to consider $\overline{\rm WMVSP}^+_R$ is 
the strong-convexity of the objective function. By Lemma~\ref{lem:d^2}, we have:
\begin{Lem}\label{lem:Lipschitz_g}
	The objective function of $\overline{{\rm WMVSP}}_R^{+}$
	is strongly convex with parameter~$2 \epsilon$.  
%	and is $\tilde L$-Lipschitz with $\tilde L = O (W mn (m+n)^{3/2})$.	
\end{Lem}
%\begin{proof}
%	This strong-convexity follows from Lemma~\ref{lem:d^2}.
%    Also the Lipschitz constant of the sum of $d^2$ terms is $O(n+m)$, 
%    which is dominated by other terms by Lemma~\ref{lem:Lipschitz_g0}
%\end{proof}

Let $g$ and $\tilde g$ denote 
the objective functions of 
$\overline{\rm WMVSP}_R$ and of $\overline{\rm WMVSP}^+_R$, respectively.
\begin{Lem}\label{lem:g_tilde_g}
	Let $z^*$ and
	$\tilde z$ 
	be minimizers of $g$ and $\tilde g$, respectively. 
	For every point~$z$, it holds that
	\begin{equation*}
		g(z) - g(z^*) \leq \tilde g(z) - \tilde g(\tilde z) + 1/2.
	\end{equation*}
\end{Lem}
\begin{proof}
	This follows from $g(z) - g(z^*) = g(z) - g(\tilde z) + g(\tilde z) - g(z^*) \leq \tilde g(z) - \tilde g(\tilde z) + \epsilon d^2(\tilde z)  + \tilde g(\tilde z) - \tilde g(z^*) + \epsilon d^2(z^*) \leq  \tilde g(z) - \tilde g(\tilde z) + 2 \epsilon (m+n)$,
%	\begin{eqnarray*}
%	&& g(z) - g(z^*) = g(z) - g(\tilde z) + g(\tilde z) - g(z^*) \\
%	&& \leq \tilde g(z) - \tilde g(\tilde z) + 
%	\epsilon d^2(\tilde z)  + \tilde g(\tilde z) - \tilde g(z^*) + 
%	\epsilon d^2(z^*) 
%	\\
%	&& \leq  \tilde g(z) - \tilde g(\tilde z) + 2 \epsilon (m+n).
%	\end{eqnarray*}
	where we use $\diam K({\cal L}_{\alpha})  = \sqrt{m_{\alpha}}$ 
	and $\diam K({\cal M}_{\beta})  = \sqrt{n_{\beta}}$ (Corollary~\ref{cor:diam}). 
\end{proof}
To apply SPPA, 
we regard the objective function $\tilde g$ as 
the sum $\sum_{i=1}^N f_i$ with $N = \mu + \nu + \mu \nu$, where $f_i$ is defined by
\begin{equation*}
f_i(z) := \left\{
\begin{array}{cl}
- C_{\alpha} d_1(x_{\alpha}) + \epsilon d^2(x_{\alpha}) & {\rm if}\ 
i = \alpha, \\
- D_{\beta} d_1(y_{\beta}) + \epsilon d^2(y_{\beta}) & {\rm if}\ 
i = \mu + \beta, \\
M \overline{R_{\alpha \beta}}(x_{\alpha}, y_{\beta}) & 
{\rm if}\ i = \mu + \nu + \alpha (\nu-1) + \beta
\end{array}\right.
\end{equation*}
for $z = (x_1,x_2,\ldots,x_{\mu},y_1,y_2,\ldots y_{\nu})$,  
$\alpha \in  \{1,2,\ldots,\mu\}$, and $\beta \in \{1,2,\ldots,\nu\}$.
\begin{Thm}\label{thm:prox}
	Let $(z_{\ell})$ be the sequence obtained by SPPA 
	applied to $\tilde g = \sum_{i=1}^Nf_i$ with $a := 1/2$.
	For $\ell = \Omega (W^8 m^9 n^9 (m+n)^{24})$, the support of $z_{\ell}$ 
contains a minimizer of WMVSP.
\end{Thm}
\begin{proof}
	We first show that each summand $f_i$ 
	is $L$-Lipschitz with
	\[ 
	L = O(W(m+n)^{5/2}).
	\]
	By Lemma~\ref{lem:Lipschitz}, the Lipschitz constant of 
		 $d_1$ is $O(m_{\alpha}^{3/2})$ on $K({\cal L}_{\alpha})$,
		 and $O(n_{\beta}^{3/2})$ on $K({\cal M}_{\beta})$.
		 By Lemma~\ref{lem:d^2}, the Lipschitz constant of $d^2$ is $O(m_{\alpha})$ on $K({\cal L}_{\alpha})$, 
		 and $O(n_{\beta})$ on $K({\cal M}_{\beta})$.
		 If $f_i = - C_{\alpha} d_1 + \epsilon d^2$ or $- D_{\beta} d_1 + \epsilon d^2$, 
		 then the Lipschitz constant of $f_i$ is $O(W(n+m))$.
		 On the other hand, 
		the Lipschitz constant of $f_i = M \overline{R_{\alpha \beta}}$ 
		is $O(W (m+n) \min \{m_{\alpha},n_{\beta}\} (m_{\alpha} + n_{\beta})^{1/2}) = O(W(m+n)^{5/2})$.

	%	 and the Lipschitz constant of 
	%	 $\overline{R_{\alpha \beta}} = 
	%	 O(\min \{ m_{\alpha}, n_{\beta}\} (m_{\alpha} + n_{\beta})^{1/2})$.
	%	 Thus the Lipschitz constant of the sum of $d_1$ terms 
	%	 is $O(W (\sum_{\alpha} {m_{\alpha}}^{3/2} + \sum_{\beta} {n_{\beta}}^{3/2})) = O(W (m^{3/2}+ n^{3/2}))$.
	%	 The  Lipschitz constant of the sum of $M \overline{R_{\alpha \beta}}$
	%	 is $O(W mn (m+n)^{3/2})$. 

By Theorem~\ref{thm:OhtaPalfia}, 
\begin{equation*}
d(z_{kN},\tilde z)^2 \leq
\frac{1}{k^{1/2}} 
\left( (n+m) + h(1/2)\frac{L^2N(N+1)}{4\epsilon^2} \right)
=  O\left(\frac{W^2 m^2n^2 (m+n)^7}{k^{1/2}} \right). 
\end{equation*}
Thus we have
\begin{eqnarray*}
&& \tilde g(z_{kN}) -\tilde g(\tilde z) \leq N L d(z_{kN},\tilde z) = 
O\left(\frac{W^2 m^2n^2 (m+n)^{6}}{k^{1/4}} \right). 
\end{eqnarray*}
Thus, for $k = \Omega (W^8 m^8 n^8 (m+n)^{24})$, it holds
	$
	\tilde g(z_{kN}) - \tilde g(\tilde z)  < 1/2$.
	By Lemma~\ref{lem:g_tilde_g}, 
we have $g(z_{kN}) - g(z^*) < 1$.
By Lemma~\ref{lem:minimizer}, the support of $z_{kN}$ contains a minimizer of WMVSP.
\end{proof}
By Lemma~\ref{lem:minimizer}, after a polynomial number of iterations,  a minimizer exists in the support of $z_{\ell}$, where 
$z_{\ell}$ should be represented as a formal sum in 
$K(\prod_{\alpha}{\cal L}_\alpha \times \prod_{\beta} \check{\cal M}_\beta)$ 
via the algorithm in Lemma~\ref{lem:product}.

Thus, our remaining task to prove Theorem~\ref{thm:main'}
is to show that the resolvent of each summand
can be computed in polynomial time.

\subsection{Computation of resolvents}
First we consider the resolvent of 
$- C_{\alpha} d_1 + \epsilon d^2$ or 
$-D_{\beta} d_1 + \epsilon d^2$.
This is an optimization problem over the orthoscheme complex of a single lattice.
Let ${\cal L}$ be a complemented modular lattice of rank $n$.
It suffices to consider the following problem.
\begin{eqnarray*}	
{\rm P1: \quad  Min}. && - C d_1({\bf 0},x) + \epsilon d({\bf 0},x)^2 + \frac{1}{2\lambda} d(x,x^0)^2 \\
{\rm s.t.} && x \in K({\cal L}),
\end{eqnarray*}
where $C \geq 0$, $\epsilon \geq 0$, $\lambda > 0$, and $x^0 \in K({\cal L})$.
\begin{Lem}
	Suppose that $x^0$ belongs to a maximal simplex $\sigma$. 
	Then the minimizer $x^*$ of P1 
	exists in $\sigma$.
\end{Lem}
\begin{proof}
	Let $x^0 = \sum_{i=0}^n \lambda_i p_i$, 
	where $\sigma$ corresponds to maximal chain $\{p_i\}$.
	Let $x^* = \sum_{i} \mu_i q_i$ be the unique minimizer of P1.
	Consider a frame ${\cal F} = \langle a_1,a_2,\ldots,a_n \rangle$ 
	containing chains $\{p_i\}$ and $\{q_i\}$.
	Let $(x^0_1,x^0_2,\ldots,x^0_n)$ and $(x^*_1,x^*_2,\ldots,x^*_n)$ be 
	${\cal F}$-coordinates of $x^0$ and $x^*$, respectively. 
	In $K({\cal F}) \simeq [0,1]^n$, the objective function of P1 is written as
	\[
	- C \sum_{i=1}^n x_i + \epsilon \sum_{i=1}^n x_i^2 + \frac{1}{2\lambda} \sum_{i=1}^n (x_i - x^0_i)^2.
	\] 
	We can assume that $p_i = a_1 \vee a_2 \vee \cdots \vee a_i$ by relabeling.
	Then $x^0_1 \geq x^0_2 \geq \cdots \geq x^0_n$.
	Suppose that $x^0_i > x^0_{i+1}$.
	Then $x^*_i \geq x^*_{i+1}$ must hold. 
	If $x^*_i < x^*_{i+1}$, then interchanging the $i$-coordinate and $(i+1)$-coordinate of 
	$x^*$ gives rise to another point in $K({\cal F})$ having a smaller objective value, contradicting the optimality of $x^*$.
	Suppose that $x^0_i = x^0_{i+1}$. 
	If  $x^*_i \neq x^*_{i+1}$, 
	then replace both $x_i^*$ and $x_{i+1}^*$ by $(x^*_i+x^*_{i+1})/2$ 
	to decrease the objective value, which is a contradiction.
	Thus $x^*_1 \geq x^*_2 \geq \cdots \geq x^*_n$.
	By~(\ref{eqn:recover}), 
	the original coordinate is written as 
	$x^* = (1- x^{*}_1) {\bf 0} + \sum_{i=1}^n (x^*_i - x^*_{i+1}) (a_{1} \vee a_2 \vee \cdots \vee a_i) = \sum_{i} (x^*_i - x^*_{i+1}) p_i$ 
	(with $x_0^* = 1$ and $x^*_{n+1} = 0$). 
	This means that $x^*$ belongs to $\sigma$.
\end{proof}
As seen in the proof, to solve P1, it suffices to choose an arbitrary frame ${\cal F}$ 
containing the chain $\{p_i\}$ for $x^0 = \sum_i \lambda_i p_i$, 
and consider the following very easy Euclidean convex optimization problem:
\begin{eqnarray*}
{\rm P1': \quad  Min}. && - C \sum_{i=1}^n x_i + \epsilon \sum_{i=1}^n x_i^2 + \frac{1}{2\lambda} \sum_{i=1}^n (x_i - x^0_i)^2 \\
{\rm s.t.} && 0 \leq x_i \leq 1 \quad (1 \leq i \leq n),
\end{eqnarray*}
where $x$ and $x^0$ are represented in the ${\cal F}$-coordinate.
Then the optimal solution $x^*$ of P1$'$ is obtained coordinate-wise. 
Namely $x^*_i$ is $0$, $1$, or $(x_i^0 + \lambda C)/(1+ 2\epsilon \lambda)$
for each $i$.

Summarizing, P1 can be solved as follows:
choose any frame ${\cal F}$ 
containing $\{ p_i\}$ (for $x' = \sum_{i} \lambda_i p_i$), 
obtain the ${\cal F}$-coordinate of $x'$,
solve P1$'$ to obtain minimizer $x^* \in [0,1]^n$, and 
recover $x^*$ in $K({\cal L})$ by~(\ref{eqn:recover}).	
\begin{Thm}\label{thm:P1}
	The resolvent of $- Cd_1 + \epsilon d^2$ 
	is computed in polynomial time.
\end{Thm}

Next we consider the computation of 
the resolvent of~$M \overline{R_{\alpha \beta}}$.
Let $U$ and $V$ be vector spaces over field ${\bf F}$ of dimensions $m$ and $n$, respectively.
Let $A: U \times V \to {\bf F}$ be a bilinear form.
Let ${\cal L}$ and ${\cal M}$ be the (complemented modular) lattices 
of all vector subspaces of vector spaces $U$ and $V$, respectively, 
where the partial order is the inclusion order.
Let $\check{\cal M}$ be the opposite lattice, which is also complemented modular.
Recall the submodular function $R: {\cal L} \times \check{\cal M} \to \ZZ$ defined by (\ref{eqn:R}), 
and let $\overline{R}: K({\cal L} \times\check{\cal M}) \to \RR$ be the Lov\'asz extension of $R$. 
For the computation of the resolvent of~$M \overline{R_{\alpha \beta}}$, 
it suffices to consider the following problem:
\begin{eqnarray*}	
	{\rm P2: \quad  Min.} && \overline{R}(x,y) +  \frac{1}{2 \lambda} ( d(x,x^0)^2 + d(y,y^0)^2) \\
	{\rm s.t.} && (x,y) \in K({\cal L}) \times K(\check{\cal M}),
\end{eqnarray*}
where $\lambda > 0$, $x^0 \in K({\cal L})$, and $y^0 \in K(\check{\cal M})$.
Recall Lemma~\ref{lem:product} for $K({\cal L} \times \check{\cal M}) \simeq 
K({\cal L}) \times K(\check{\cal M})$.
As in the case of P1, we reduce P2 to a convex optimization 
over $[0,1]^m \times [0,1]^n$
by taking a special frame $\langle e_1,e_2,\ldots,e_m,f_1,f_2,\ldots,f_n\rangle$ of ${\cal L} \times \check{\cal M}$.

For $X \in {\cal L}$, let $X^{\bot}$ denote the subspace in $\check{\cal M}$ defined by
\begin{equation}
X^{\bot} := \{ y \in V \mid A(x,y) = 0 \  (x \in X) \}.
\end{equation}
Namely $X^{\bot}$ is the orthogonal subspace of $X$ with respect to the bilinear form $A$.
For $Y \in \check{\cal M}$, let $Y^{\bot} \in {\cal L}$ be defined similarly.

Let $r := \rank A$.
An {\em $A$-orthogonal frame} ${\cal F} = \langle e_1,e_2,\ldots,e_m,f_1,f_2,\ldots,f_n\rangle$
is a frame of ${\cal L} \times \check{\cal M}$ satisfying the following conditions:
\begin{itemize}
\item $\langle e_1,e_2,\ldots,e_m \rangle$ is a frame of ${\cal L}$.
\item $\langle f_1,f_2,\ldots,f_n\rangle$ is a frame of $\check{\cal M}$.
\item $e_{r+1} \vee e_{r+2} \vee \cdots \vee e_m = V^{\bot}$.
\item $f_1 \vee f_2 \vee \cdots \vee f_r = U^{\bot}$ ($\Leftrightarrow$ $f_1 \cap f_2 \cap \cdots \cap f_r = U^{\bot}$ ).
\item $f_i = {e_i}^{\bot}$ for $i=1,2,\ldots,r$.
\end{itemize}
For an $A$-orthogonal frame ${\cal F} = \langle e_1,e_2,\ldots,e_m,f_1,f_2,\ldots,f_n\rangle$, 
the Lov\'asz extension $\overline{R}$ of $R$
takes a much simpler form on $K({\cal F})$ as follows, where 
the proof  is given in Section~\ref{subsec:proof}.
\begin{Thm}\label{thm:Lovasz}
	Let  ${\cal F} = \langle e_1,e_2,\ldots,e_m,f_1,f_2,\ldots,f_n\rangle$ be an $A$-orthogonal frame.
	The restriction of the Lov\'asz extension $\overline{R}$ to 
	$K({\cal F}) \simeq [0,1]^m \times [0,1]^n$  is written as
	\begin{equation}
	\overline{R} (x,y) = \sum_{i=1}^r \max\{0, x_i - y_i\} \quad (x \in [0,1]^m, y \in [0,1]^n),
	\end{equation}
	where $(x_1,x_2,\ldots,x_m)$ is the $\langle e_1,e_2,\ldots,e_m \rangle$-coordinate of $x$
	and $(y_1,y_2,\ldots,y_n)$ is the $\langle f_1,f_2,\ldots,f_n \rangle$-coordinate of $y$.
\end{Thm}
The main ingredient in solving P2 is the following, where 
the proof is given is Section~\ref{subsec:proof}.
Figure~\ref{fig:A-ortho} illustrates an $A$-orthogonal frame
in this theorem.
 \begin{Thm}\label{thm:R(x,y)}
	 Let ${\cal X}$ and ${\cal Y}$ 
	 be maximal chains corresponding to 
	 maximal simplices containing $x^0$ and $y^0$, respectively.
	 \begin{itemize}
	 	\item[{\rm (1)}] There exists an $A$-orthogonal frame ${\cal F} = \langle e_1,e_2,\ldots,e_m,f_1,f_2,\ldots,f_n\rangle$ satisfying
	 	\begin{equation}\label{eqn:XuYbot}
	 	{\cal X} \cup {\cal Y}^{\bot} \subseteq 
	 	\langle e_1,e_2,\ldots,e_m \rangle,\   {\cal X}^{\bot} \cup {\cal Y} \subseteq 
	 	\langle f_1,f_2,\ldots,f_n \rangle,  
	 	\end{equation}
	 	in which such a frame can be found in polynomial time.
	 	\item[{\rm (2)}] For an $A$-orthogonal frame ${\cal F}$ satisfying $(\ref{eqn:XuYbot})$, the minimizer $(x^*,y^*)$ of P2
	 	exists in $K({\cal F})$.
	 \end{itemize}
\end{Thm}
	\begin{figure}[t]
		\begin{center}
			\includegraphics[scale=0.6]{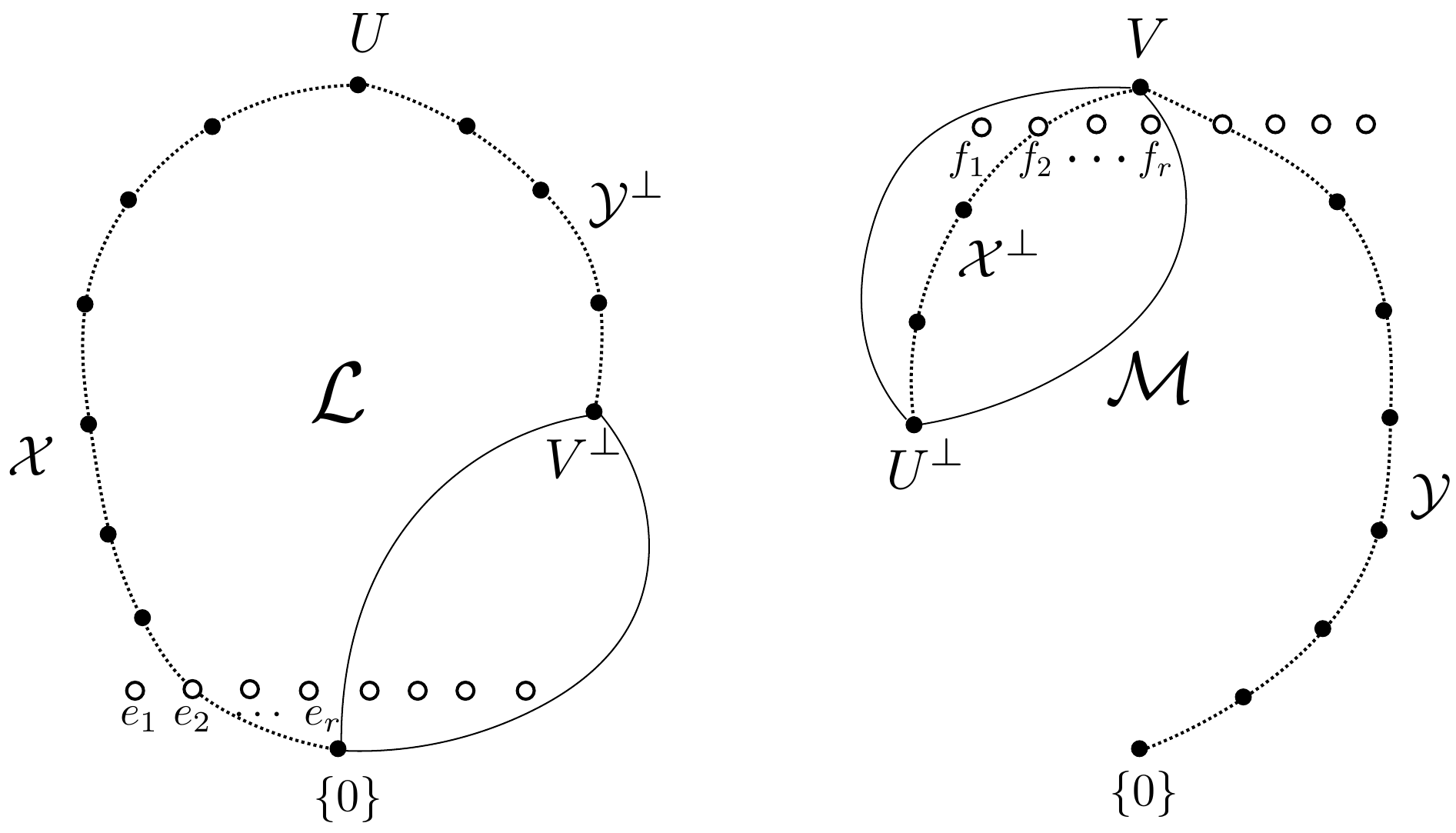}
			\caption{$A$-orthogonal frame in Theorem~\ref{thm:R(x,y)}}
			\label{fig:A-ortho}
		\end{center}
	\end{figure}\noindent

Assume Theorems~\ref{thm:Lovasz} and \ref{thm:R(x,y)}. 
For an $A$-orthogonal frame satisfying~(\ref{eqn:XuYbot}), the problem P2 is equivalent to
\begin{eqnarray*}	
	{\rm P2': \quad  Min.} && \sum_{i=1}^r \max\{0,x_i - y_i \} + \frac{1}{2 \lambda} \left\{\sum_{i=1}^m (x_i -x^0_i)^2 + \sum_{i=1}^n (y_i - y^0_i)^2 \right\} \\
	{\rm s.t.} && 0 \leq x_i \leq 1 \quad (0 \leq i \leq m), \\
	&&  0 \leq y_i \leq 1 \quad (0 \leq i \leq n).
\end{eqnarray*}
Again this problem is easily solved coordinate-wise.
Obviously $x^*_i = x^0_i$ and $y^*_i = y^0_i$ for $i > r$.
For $i \leq r$,  
$(x^*_i,y^*_i)$ 
is the minimizer of the following $2$-dimensional problem:
\begin{eqnarray*}
	{\rm Min.} && 
\max\{0,x_i - y_i \} + \frac{1}{2 \lambda} \left\{(x_i -x^0_i)^2 + (y_i - y^0_i)^2 \right\}\\
{\rm s.t.} && 0 \leq x_i \leq 1,\ 0 \leq y_i \leq 1. 
\end{eqnarray*}
Obviously this can be solved in constant time.

Thus we can solve P2 as follows. Choose 
an $A$-orthogonal frame ${\cal F}$ satisfying (\ref{eqn:XuYbot}), 
solve P2$'$ to obtain the minimizer $(x^*,y^*) \in [0,1]^m \times [0,1]^n$,  and recover $(x^*,y^*)$ in $K({\cal L}) \times K({\cal M})$.  
\begin{Thm}\label{thm:P2}
	The resolvent of $\overline{R}$ is computed in polynomial time.
\end{Thm}
Combining Theorems~\ref{thm:prox}, \ref{thm:P1}, and \ref{thm:P2}, 
the proof of Theorem~\ref{thm:main'} is completed.

\begin{Rem}\label{rem:bit}
	In the above SPPA, 
	the required bit-length for coefficients of $z \in K({\cal L} \times \check{\cal M})$ 
	is bounded polynomially in $n,m,W$. Indeed,
	the transformation between the original coordinate and an ${\cal F}$-coordinate
	corresponds to multiplying a triangular matrix consisting of $\{0,1,-1\}$ elements; see (\ref{eqn:recover}). 
	In each iteration $k$,
	the optimal solution of quadratic problem P1$'$ or P2$'$ 
	is obtained by 
	adding (fixed) rational functions in $n,m,W,k$ 
	to (current points) $x_i^0, y_i^0$ 
	and multiplying a (fixed) $2 \times 2$ rational matrix in $n,m,W,k$.   
	Consequently the bit increase is bounded as required.

	On the other hand, the bit-length estimation for a basis 
	of a vector subspace appearing in the algorithm is not clear. 
\end{Rem}

\begin{Rem}\label{rem:ncrank}
Our algorithm is easily adapted to compute nc-rank in the same time complexity; 
see Introduction for nc-rank.
Indeed, by (\ref{eqn:weakdual2}), it suffices to solve 
\begin{eqnarray*}
{\rm Min.} && - \dim X - \dim Y + M \sum_{i=1}^{N} R^{A_i}(X,Y) \\
{\rm s.t.} && X \in {\cal L},\ Y \in \check{\cal M},
\end{eqnarray*}
where ${\cal L}$ and ${\cal M}$ are the lattices of all vector subspaces 
of ${\bf F}^m$ and ${\bf F}^n$, respectively, and $M := n+m+1$. 
As above, we may consider the following perturbed continuous relaxation:
\begin{eqnarray*}
{\rm Min.} && - d_1(x) - d_1(y) + M \sum_{i=1}^{N} \overline{R^{A_i}}(x,y) + 
\epsilon (d^2(x) + d^2(y)) \\
{\rm s.t.} && (x,y) \in K({\cal L} \times \check{\cal M}),
\end{eqnarray*}
where $\epsilon := 1/4(n+m)$.
In the setting of 
$f_i(x,y) := M \overline{R^{A_i}}(x,y)$ for $(1 \leq i \leq N)$, 
$- d_1(x) + \epsilon d^2(x)$ for $i=N+1$, and $- d_1(y) + \epsilon d^2(y)$ for $i=N+2$, 
the SPPA and the above analysis are applicable.  
\end{Rem}

\subsection{Proof}\label{subsec:proof}
We start with basic properties of $(\cdot)^{\bot}$, which follow from elementary linear algebra.
\begin{Lem}\label{lem:bot}
	\begin{itemize}
		\item[{\rm (1)}] If $X \subseteq X'$, then $X^{\bot} \supseteq {X'}^{\bot}$ and
		$
		r(X^\bot) - r({X'}^{\bot}) \leq r(X') - r(X).
		$
		\item[{\rm (2)}] $(X + X')^{\bot} = X^{\bot} \cap {X'}^{\bot}$. 
		\item[{\rm (3)}] $X^{\bot \bot} \supseteq X$.
		\item[{\rm (4)}] $X^{\bot \bot \bot} = X^{\bot}$.
	\end{itemize}
\end{Lem}

Next we give an alternative expression of $R$ by using $(\cdot)^{\bot}$.
Let $\check r$ be the rank function of $\check{\cal M}$. 
Namely $\check r(Y) = m - \dim Y$. 
\begin{Lem}\label{lem:formula_R}
	$R(X,Y) = r(X) - r(X \wedge Y^{\bot}) = \check r(Y \vee X^{\bot}) - \check r(Y) (= r(Y) - r(Y \cap X^{\bot}))$.
\end{Lem}
\begin{proof}
	Consider bases $\{a_1,a_2,\ldots,a_{\ell}\}$ of $X$ and $\{ b_1,b_2,\ldots,b_{\ell'}\}$ of $Y$.
	We can assume that
	$\{a_{k+1},a_{k+2},\ldots,a_{\ell}\}$ is a base of $X \cap Y^{\bot}$.
	Consider the matrix representation $(A(a_i,b_j))$ of $A|_{X \times Y}$ 
	with respect to the bases.
	Since $A(a_{i},Y) = \{0\}$ for $i > k$, 
	the submatrix of $k+1,k+2,\ldots,\ell$-th rows 
	is a zero matrix.
	On the other hand, the submatrix of $1,2,\ldots k$-th rows
	must have the row-full rank $k$. 
%	Otherwise we can construct a nonnegative combination $a \in X$ of $a_1,a_2,\ldots,a_k$ such that
%	$a, a_{k+1},a_{k+2},\ldots,a_{l}$ are independent and
%	$A(a',Y) = \{0\}$. This contradicts the fact that $a_{k+1},a_{k+2},\ldots,a_{l}$ is a base of $X \cap Y^{\bot}$.
	Thus the rank $R(X,Y)$ of $(A(a_i,b_j))$ is $k = \ell - (\ell-k) = r(X) - r(X \wedge Y^{\bot})$.
	The same consideration shows the second equality.
\end{proof}

\paragraph{Proof of Theorem~\ref{thm:Lovasz}.}
An $A$-orthogonal frame 
$\langle e_1,e_2,\ldots,e_m,f_1,f_2,\ldots,f_n \rangle 
= \langle e_1,e_2,\ldots,e_m \rangle \times \langle f_1,f_2,\ldots,f_n \rangle$ 
is naturally identified with Boolean lattice 
$2^{\{1,2,\ldots,m\}} \times 2^{\{1,2,\ldots,n\}}$ 
(by $(X,Y) \mapsto \bigvee_{i \in X} e_i \vee \bigvee_{j \in Y} f_j$). 
Then $r = |\cdot|$, $\check{r} = |\cdot|$, and  $\vee = \cup$.
Notice that ${e_i}^{\bot} = f_i$ if $i \leq r$ 
and ${e_i}^{\bot} = V$ if $i > r$.
The latter fact follows from 
$e_i \subseteq V^{\bot} \Rightarrow {e_i}^{\bot} \supseteq V^{\bot \bot} \supseteq V$.
This implies that
$X^{\bot} = X \cap \{1,2,\ldots,r\}$ for $X \in 2^{\{1,2,\ldots,m\}}$.
By Lemma~\ref{lem:formula_R}, we have
\[
	R(X,Y)  =  |Y \cup (X \cap \{1,2,\ldots,r\}) | - |Y| =  |(X \setminus Y) \cap \{1,2,\ldots,r\}|.
\]
%\begin{eqnarray*}
%	R(X,Y) & = & |Y \cup (X \cap \{1,2,\ldots,r\}) | - |Y| \\
%	&= & |(X \setminus Y) \cap \{1,2,\ldots,r\} | \quad 
%	((X,Y) \in 2^{\{1,2,\ldots m\}} \times 2^{\{1,2,\ldots,n\}}).
%\end{eqnarray*} 
Identify $2^{\{1,2,\ldots,m\}} \times 2^{\{1,2,\ldots,n\}}$ with $\{0,1\}^m \times \{0,1\}^n$ 
by $(X,Y) \mapsto (1_{X}, 1_{Y})$.
Then $R$ is also written as
\begin{equation*}
R(x,y) = \sum_{i=1}^r \max \{0, x_i - y_i\} \quad ((x,y) \in \{0,1\}^m \times \{0,1\}^n).
\end{equation*}
The Lov\'asz extension $\overline R$ is equal to
the function obtained from $R$ by
extending the domain to $[0,1]^m \times [0,1]^n$.

\paragraph{Proof of Theorem~\ref{thm:R(x,y)}~(1).}
By Lemma~\ref{lem:frame}, 
we can find (in polynomial time) a frame $\langle e_1,e_2,\ldots,e_m \rangle$ 
containing two chains ${\cal X}$ and ${\cal Y}^{\bot}$.
Suppose that ${\cal X} = \{X_i\}_{i=0}^m$ and ${\cal Y} = \{Y_i\}_{i=0}^n$.
We can assume that $e_{r+1} \vee e_{r+2} \vee \cdots \vee e_{m} = {Y_0}^{\bot} = V^{\bot}$ by relabeling.
Let $f_i := {e_i}^{\bot}$ for $i=1,2,\ldots,r$.
Then $f_1 \vee f_2 \vee \cdots \vee f_r = U^{\bot}$ holds.
Indeed, by Lemma~\ref{lem:bot},  we have $U^{\bot} =  (e_1 \vee e_2 \vee \cdots \vee e_m)^{\bot} = {e_1}^{\bot} \vee {e_2}^{\bot} \vee \cdots \vee {e_m}^{\bot} = f_1 \vee f_2 \vee \cdots f_r \vee V \vee \cdots \vee V = f_1 \vee f_2 \vee \cdots \vee f_r$. 
%where $e_{r+i} \subseteq V^{\bot}$ implies 
%$e_{r+i}^{\bot} \supseteq V^{\bot \bot} \supseteq V$, and $e_{r+i}^{\bot} = V$.
%

Consider the chain ${\cal Y}^{\bot \bot}$ in $\check{\cal M}$.
Then ${\cal Y}^{\bot \bot} \subseteq \langle f_1,f_2,\ldots,f_r \rangle$.
Indeed, each $Y_i^{\bot}$ is a join of a subset of $e_1,e_2,\ldots,e_m$.
Taking $(\cdot)^\bot$ as above, 
$Y_i^{\bot \bot}$ is represented as a join of a subset of $f_1,f_2,\ldots,f_r$.
Consider a consecutive pair $Y_{i-1}, Y_{i}$ in ${\cal Y}$.
Consider ${Y_{i-1}}^{\bot \bot}$ and ${Y_{i}}^{\bot \bot}$.
Then, by Lemma~\ref{lem:bot}~(3), 
${Y_{i-1}}^{\bot \bot} \preceq Y_{i-1}$ and ${Y_{i}}^{\bot \bot} \preceq Y_{i}$.
Suppose that ${Y_{i-1}}^{\bot \bot} \neq {Y_{i}}^{\bot \bot}$.
Then ${Y_{i-1}}^{\bot \bot} \prec_1 {Y_{i}}^{\bot \bot}$ (by Lemma~\ref{lem:bot}~(1)).
Thus for some $f_j$ $(1 \leq j \leq r)$, 
it holds ${Y_{i}}^{\bot \bot} = f_j \vee {Y_{i-1}}^{\bot \bot}$.
Here $f_j \not \preceq Y_{i-1}$ must hold.
Otherwise ${Y_{i-1}}^{\bot \bot} \succeq {f_j}^{\bot \bot}  = {e_j}^{\bot \bot \bot} = f_j$, 
which contradicts ${Y_{i-1}}^{\bot \bot} \prec_1 {Y_{i}}^{\bot \bot} = f_j \vee {Y_{i-1}}^{\bot \bot}$.
Thus $Y_{i} = Y_{i-1} \vee f_j$.
Therefore, for each $i$ with ${Y_{i-1}}^{\bot \bot} = {Y_{i}}^{\bot \bot}$, 
we can choose an atom $f$ with $Y_i = f \vee Y_{i-1}$ to add to $f_1,f_2,\ldots,f_r$, 
and obtain a required frame 
$\langle f_1,f_2,\ldots f_n \rangle$ (containing ${\cal X}^{\bot}$ and ${\cal Y}$).

\paragraph{Proof of Theorem~\ref{thm:R(x,y)}~(2).}
The proof is long. An outline of the proof with an intuition is explained as follows:
\begin{itemize}
	\item Imagine the geodesic $\gamma$ emanating from $(x^0,y^0)$ to the minimizer $(x^*,y^*)$ of P2.
	\item In the {\em generic} case, 
	      the geodesic meets maximal simplices $K_0,K_1,K_2,\ldots,K_{\ell} \subseteq K({\cal L} \times \check{\cal M})$ in order
	      so that $K_{i} \cap K_{i+1}$ has dimension $n+m-1$. This yields
	      a sequence ({\em gallery}) of corresponding maximal chains ${\cal C}_0, {\cal C}_1,{\cal C}_2, \ldots, {\cal C}_{\ell}$ in ${\cal L} \times \check{\cal M}$.  
	\item This gallery must have a special property (Lemma~\ref{lem:generic}), called the {\em $A$-orthogonality}, which we will introduce.
	\item On the other hand, any $A$-orthogonal gallery ${\cal C}_0, {\cal C}_1,{\cal C}_2, \ldots, {\cal C}_{\ell}$
	belongs to the product of sublattices generated by 
	${\cal X} \cup {\cal Y}^{\bot}$ and ${\cal Y} \cup {\cal X}^{\bot}$, where
	${\cal X}$ and ${\cal Y}$ are the {\em projections} of the initial ${\cal C}_0$ to ${\cal L}$ 
	and to $\check{\cal M}$, respectively (Lemma~\ref{lem:A-ortho_gallery}).
	\item In the generic case, the above imply that $(x^*,y^*)$ belongs to 
	the product of sublattices generated by ${\cal X} \cup {\cal Y}^{\bot}$ and ${\cal Y} \cup {\cal X}^{\bot}$, where ${\cal X}$ and ${\cal Y}$ are the supports of $x^0$ and $y^0$, respectively. This implies Theorem~\ref{thm:R(x,y)}~(2).
	\item By perturbation, we remove the genericity assumption (Lemma~\ref{lem:remove}). 
\end{itemize}
To formulate the $A$-orthogonality, 
we start with a general lemma of a modular lattice.
\begin{Lem}\label{lem:2-interval}
	Let ${\cal L}$ be a modular lattice.
	Let $p,p' \in {\cal L}$ with $p \prec_2 p'$, 
	and let ${\cal C}$ be a chain such that
	\[
	{\cal C}_1 := \{ q\in {\cal C} \mid r(p' \wedge q) - r(p \wedge q) = 1 \}
	\]
	is nonempty.
	Then  there is a unique element $u^*$ with $p \prec_1 u^* \prec_1 p'$ such that
		for all  $q \in {\cal C}_1$ and all $u \neq u^*$ with  $p \prec_1 u \prec_1 p'$ it holds 
		\begin{eqnarray*}
			&& r(u^*) - r(u^* \wedge q) = r(u)- r(u \wedge q) - 1, \\
			&& r(q) - r(u^* \wedge q) = r(q)- r(u \wedge q) - 1,
		\end{eqnarray*}
		where $u^*$ is equal to $p \vee (p' \wedge q) (= p' \wedge (p \vee q))$ for 
		all $q \in {\cal C}_1$.
\end{Lem}
Intuitively speaking, this $u^*$ is the element {\em closest} to ${\cal C}$ 
among elements $u$ with $p \prec_1 u \prec_1 p'$; 
see Figure~\ref{fig:g}.
The element $u^*$ plays an important role, and is denoted by $g(p,p',{\cal C})$. 
	\begin{figure}[t]
		\begin{center}
			\includegraphics[scale=0.6]{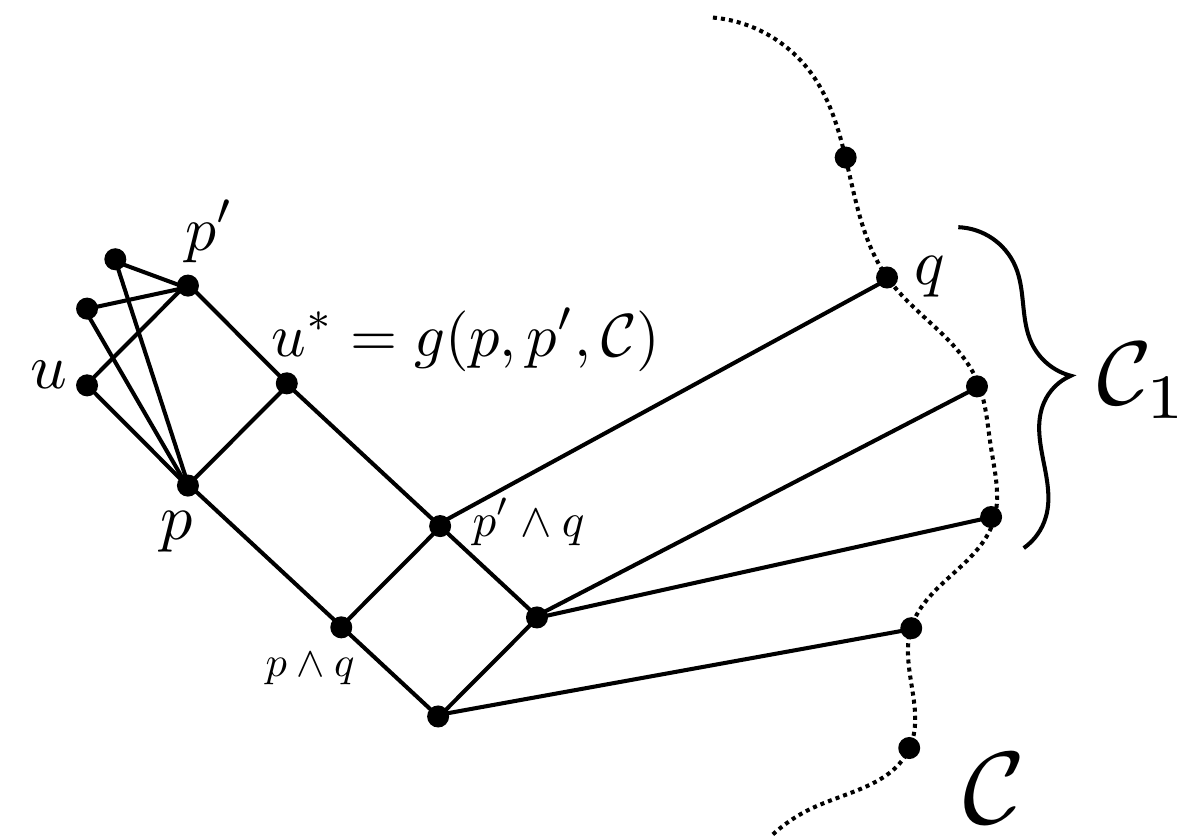}
			\caption{$g(p,p',{\cal C})$}
			\label{fig:g}
		\end{center}
	\end{figure}\noindent
\begin{proof}
	Let $q \in {\cal C}_1$.
	Then $p \wedge q \prec_1 p' \wedge q$.
	Let $u^* := p \vee (p' \wedge q)$. 
	Then $p \prec_1 u^* \prec_1 p'$ and $p' \wedge q = u^* \wedge q$.  Thus 
	\begin{equation*}
	r(u^*) - r(u^* \wedge q) = r(p') - r(p' \wedge q) - 1.
	\end{equation*}
	Consider $u \neq u^*$ with $p \prec_1 u \prec_1 p'$.
	If $p' \wedge q \preceq u$, 
	then $p' \wedge q \preceq u \wedge u^* = p$ and $p' \wedge q = p \wedge q$, which 
	contradicts $p \wedge q \prec_1 p' \wedge q$.
	Thus $p' \wedge q \not \preceq u$.
	Therefore $u \vee (p' \wedge q) = p'$, and
	$p \wedge q \preceq u \wedge q \prec_1 p' \wedge q$. 
	Necessarily $p \wedge q = u \wedge q$. Hence we have
	\begin{equation*}
	r(u) - r(u \wedge q) = r(p') - r(p' \wedge q) = r(u^*)  - r(u^* \wedge q) + 1,
	\end{equation*}
	which also implies the second equality by $r(u) = r(u^*)$.
%Also  $r(q) - r(u^* \wedge q) = r(q)- r(u \wedge q) - 1$ follows from
%$p' \wedge q = u^* \wedge q$, $p \wedge q = u \wedge q'$ and $q \in {\cal C}_0$.

	Next we show that $u^*$ is independent of $q$.
	Consider another $q' \in {\cal C}_1$.
	We may assume that $q \prec q'$. Let $u^{**} := p \vee (p' \wedge q')$.
	If $u^*$ and $u^{**}$ are different, 
	then $p' \wedge q  \preceq u^* \wedge u^{**} = p$; 
	this is a contradiction (to $p \wedge q \prec_1 p' \wedge q)$.
\end{proof}
In the case of a Boolean lattice, $g$ is simply described as follows:
\begin{Lem}\label{lem:g_Boolean}
	Suppose that ${\cal L}$ is a Boolean lattice $2^{\{1,2,\ldots,m\}}$.
	For $X,X' \subseteq \{1,2,\ldots, m\}$ 
	with $X' \setminus X = \{a,b\}$ and a chain ${\cal C}$,  
	it holds $g(X,X',{\cal C}) = X \cup \{a\}$ if and only 
	if ${\cal C}$ has a member $Z$ with $a \in Z \not\ni b$.
\end{Lem}
We say that ${\cal C}$ contains $a$ {\em before} $b$ if 
${\cal C}$ has a member $Z$ with $a \in Z \not\ni b$.

Let ${\cal L},{\cal M}$ and $A$ be as before.
Let ${\cal C} = \{ (X_i,Y_i)\}_{i=0}^{n+m}$ be a maximal chain of ${\cal L} \times \check{\cal M}$.
Then it holds that
\begin{eqnarray*}
(\{ 0\},V) = (X_0,Y_0) \prec (X_1,Y_1) \prec \cdots \prec (X_{m+n},Y_{m+n}) = (U, \{0\}),
\end{eqnarray*}
where  
\begin{equation*}
	X_{i-1} = X_i,\ Y_{i-1} \prec_1 Y_i\quad {\rm or}\quad X_{i-1} \prec_1 X_i,\ Y_{i-1} = Y_i.
\end{equation*}
for each $i \in \{1,2,\ldots,m+n\}$.
Let ${\cal X}_{\cal C}$ denote the maximal chain of 
${\cal L}$ obtained by projecting ${\cal C}$ to ${\cal L}$:
\[
{\cal X}_{\cal C} := \{ X \in {\cal L} \mid \exists Y \in \check{\cal M}: (X,Y) \in {\cal C} \}.
\]
Similarly, let ${\cal Y}_{\cal C}$ denote the maximal chain of $\check{\cal M}$ defined by
\[
{\cal Y}_{\cal C} := \{ Y \in \check{\cal M} \mid \exists X \in {\cal L}: (X,Y) \in {\cal C} \}.
\]

Let ${\cal C}' = \{(X_i',Y_i')\}$ be another maximal chain of ${\cal L} \times \check{\cal M}$.
Two chains ${\cal C}$ and ${\cal C'}$ are said to be {\em adjacent} 
if $|{\cal C} \cap {\cal C}'| = m+ n-1$.
If ${\cal C}$ and ${\cal C}'$ are adjacent, then there uniquely exists an index 
$i \in \{1,2,\ldots,m+n\}$ 
such that $(X_j,Y_j) = (X_j',X_j')$ holds for all $j \in \{0,1,2,\ldots,m+n\} \setminus \{i\}$ 
and one of the following holds for $(X_i,Y_i) \neq (X_i',Y_i')$:
\begin{itemize}
	\item[(0)] $X_{i-1} \prec_1 X_{i+1}$, $Y_{i-1} \prec_1 Y_{i+1}$, and 
	$\{(X_{i},Y_{i}),(X'_{i},Y'_i)\} = \{(X_{i-1},Y_{i+1}),(X_{i+1},Y_{i-1})\}$.  
	\item[(1)] $X_{i-1} \prec_1 X_i \prec_1 X_{i+1}$, 
	            $X_{i-1} \prec_1 X'_i \prec_1 X_{i+1}$, and $Y_{i-1} = Y_i = Y_i' = Y_{i+1}$.
	\item[(2)] $Y_{i-1} \prec_1 Y_i \prec_1 Y_{i+1}$, 
	$Y_{i-1} \prec_1 Y'_i \prec_1 Y_{i+1}$, and $X_{i-1} = X_i = X_i' = X_{i+1}$.
\end{itemize} 
${\cal C}$ and ${\cal C}'$ are said to be {\em $0$-adjacent} 
if (0) holds, {\em ${\cal L}$-adjacent} if (1) holds, 
and {\em ${\cal M}$-adjacent} if (2) holds.
%Observe that 
%if ${\cal C}$ and ${\cal C}'$ are {\em ${\cal L}$-adjacent} 
%then ${\cal X}_{\cal C}$ and ${\cal X}_{{\cal C}'}$ are adjacent.

Also ${\cal C}$ and ${\cal C}'$ are said to be 
{\em $A$-orthogonally ${\cal L}$-adjacent} from ${\cal C}$ to ${\cal C}'$ if (1) holds with
\begin{equation*}
X'_i = g(X_{i-1},X_{i+1},{{\cal Y}_{\cal C}}^{\bot}),
\end{equation*} 
and {\em $A$-orthogonally ${\cal M}$-adjacent} from ${\cal C}$ to ${\cal C}'$ if (2) holds with
\begin{equation*}
Y'_i = g(Y_{i-1},Y_{i+1},{{\cal X}_{\cal C}}^{\bot}).
\end{equation*}
Intuitively speaking, if ${\cal C}$ and ${\cal C}'$ are 
$A$-orthogonally ${\cal L}$-adjacent from ${\cal C}$ to ${\cal C}'$, 
then the transition from ${\cal C}$ to ${\cal C}'$ 
is close to ${{\cal Y}_{\cal C}}^{\bot}$ (with nonincreasing $R$).
%Indeed, by Lemmas~\ref{lem:formula_R} and \ref{lem:2-interval}, we have
%\begin{Lem}
%	If two maximal chains ${\cal C} = \{(X_i,Y_i)\}$ and 
%	${\cal C}' = \{(X_i',Y_i')\}$ are $A$-orthogonally ${\cal L}$- or ${\cal M}$-adjacent from ${\cal C}$ to ${\cal C}'$, then it holds that
%	\begin{equation}
%	R(X'_i,Y'_i) \leq R(X_i,Y_i) \quad (i = 0,1,2,\ldots,n+m)
%	\end{equation}
%\end{Lem}
%

A sequence $({\cal C}_0,{\cal C}_1,\ldots,{\cal C}_{\ell})$ is 
called a {\em gallery} if for each $i \in \{1,2,\ldots,\ell\}$, 
${\cal C}_{i-1}$ and ${\cal C}_i$ are adjacent, and 
is called an {\em $A$-orthogonal gallery} if for each $i \in \{1,2,\ldots,\ell\}$,
${\cal C}_{i-1}$ and ${\cal C}_i$ are $0$-adjacent,  
or $A$-orthogonally ${\cal L}$- or ${\cal M}$-adjacent from ${\cal C}_{i-1}$ to ${\cal C}_i$.  
\begin{Lem}\label{lem:A-ortho_gallery}
	For an $A$-orthogonal gallery $({\cal C} = {\cal C}_0,{\cal C}_1,\ldots,{\cal C}_{\ell})$, 
	it holds
	\[
	{\cal C}_{\ell} \subseteq \langle {\cal X}_{\cal C} \cup {{\cal Y}_{\cal C}}^\bot \rangle \times \langle {{\cal X}_{\cal C}}^{\bot} \cup {\cal Y}_{\cal C} \rangle,
	\]
	where $\langle {\cal Z} \rangle$ denotes the sublattice generated by ${\cal Z}$.
\end{Lem}
\begin{proof}
Let ${\cal X}_k := {\cal X}_{{\cal C}_k}$ and ${\cal Y}_k :={\cal Y}_{{\cal C}_k}$.
	Since ${\cal C}_k \subseteq \langle {\cal X}_k \cup 
	{{\cal Y}_{k}}^\bot \rangle \times \langle {{\cal X}_{k}}^{\bot} \cup {\cal Y}_{k} \rangle$, 
    It suffices to show 
    \[
    \langle {\cal X}_{k+1} \cup {{\cal Y}_{k+1}}^\bot \rangle
    \times \langle {{\cal X}_{k+1}}^{\bot} \cup {\cal Y}_{k+1} \rangle \subseteq 
    \langle {\cal X}_{k} \cup 	{{\cal Y}_{k}}^\bot \rangle \times 
    \langle {{\cal X}_{k}}^{\bot} \cup {\cal Y}_{k} \rangle.
    \] 
	It is obvious when
	${\cal C}_k$ and ${\cal C}_{k+1}$ are $0$-adjacent, 
	since ${\cal X}_{k} = {\cal X}_{k+1}$ 
	and ${\cal Y}_{k} = {\cal Y}_{k+1}$. 
	We may assume that ${\cal C}_k$ and ${\cal C}_{k+1}$ are ${\cal L}$-adjacent.
	It suffices to that 
	$X_i' \in {\cal X}_{k+1} \setminus {\cal X}_{k}$ 
belongs to $\langle {\cal X}_{k} \cup {{\cal Y}_k}^\bot \rangle$, and 
	${X_i'}^{\bot}$ belongs 
	to $\langle {{\cal X}_k}^{\bot} \cup {\cal Y}_{k} \rangle$.
	The former claim follows from Lemma~\ref{lem:2-interval} that
	$X_i'$ is represented 
	as $X_{i-1} \vee (X_{i+1} \wedge {Y_j}^{\bot})$ for some ${Y_j}^{\bot} \in {{\cal Y}_k}^\bot$.
	
	We show the latter claim. 
	By ${X_{i-1}}^{\bot} \preceq {X'_{i}}^{\bot} \preceq {X_{i+1}}^{\bot}$ 
	and $r({X_{i-1}}^{\bot}) - r({X_{i+1}}^{\bot}) \leq r(X_{i+1})- r(X_{i-1}) = 2$ (Lemma~\ref{lem:bot}), 
	it suffices to consider the case where ${X_{i-1}}^{\bot} \prec_1 {X'_{i}}^{\bot} \prec_1 {X_{i+1}}^{\bot}$. 
By $r(X'_i) - r(X_i' \wedge {Y_j}^{\bot}) = r(X_i) - r(X_i \wedge {Y_j}^{\bot}) -1$ and Lemma~\ref{lem:formula_R}, it holds.
\begin{equation*}
R(X'_i,Y_j) = R(X_i,Y_j) - 1.
\end{equation*}
This in turn implies that $r(Y_j) - r(Y_j \cap {X'_i}^{\bot}) 
= r(Y_j) - r(Y_j \cap {X_i}^{\bot}) - 1$.
	By Lemma~\ref{lem:2-interval},  ${X_{i}'}^{\bot}$ 
	must be $g({X_{i-1}}^{\bot}, {X_{i+1}}^{\bot},{\cal Y}_{k})$, 
	and belongs to $\langle {{\cal X}_k}^{\bot} \cup {\cal Y}_{k} \rangle$ as above.
\end{proof}

For a geodesic $[z,z'] \subseteq K({\cal L} \times \check{\cal M})$ and $t \in [0,1]$,
let $z^t := (1-t) z + t z'$, and
let $K_t$ denote the simplex containing $z^t$
as its relative interior.
The collection $\{ K_t\}_{t \in [0,1]}$ of simplices is finite, 
since $[z,z']$ belongs to (finite complex) $K({\cal F}) \simeq [0,1]^{n+m}$ 
for some frame ${\cal F}$.
A geodesic $[z,z']$ is said to be {\em generic} 
if $K_0$ has dimension $n+m$, and $K_t$ 
has dimension $n+m$ or $n+m-1$ for $t \in (0,1)$.
A generic geodesic $[z,z']$ gives rise to 
a gallery $({\cal C}_0,{\cal C}_1,\ldots, {\cal C}_{\ell})$ as follows.
Let ${\cal C}_0$ be a maximal chain 
corresponding to the simplex $K_0$ containing $z$ as its interior. 
For some $t_1 > 0$, 
the point $z^{t_1}$ reaches the boundary of $K_0$, 
which is a face of $K_0$ having dimension $n+m-1$.
For $t \in (t_1,t_2)$ for some $t_2 > t_1$, the point 
$z^t$ lies on the next maximal simplex $K_1$ adjacent to $K_0$.
Let ${\cal C}_1$ denote the maximal chain corresponding to $K_1$.
Then ${\cal C}_0$ and ${\cal C}_1$ are adjacent.
As $t \rightarrow 1$, 
we obtain a gallery $({\cal C}_0,{\cal C}_1,\ldots,{\cal C}_{\ell})$.
The main lemma is the following.
\begin{Lem}\label{lem:generic}
	Let $(x^*,y^*)$ be the minimizer of P2.
	If geodesic $[(x^0,y^0),(x^*,y^*)]$ is generic, 
	then the corresponding gallery $({\cal C}_0,{\cal C}_1,\ldots,{\cal C}_{\ell})$ is $A$-orthogonal.
\end{Lem}
In particular, if $[(x^0,y^0),(x^*,y^*)]$ is generic, 
then the chain ${\cal C}_{\ell}$ including the support of the minimizer $(x^*,y^*)$ belongs to 
the sublattice 
$\langle {\cal X} \cup  {\cal Y}^\bot \rangle \times \langle {\cal X}^{\bot} \cup {\cal Y} \rangle$
(by Lemma~\ref{lem:A-ortho_gallery}), which belongs to
an $A$-orthogonal frame ${\cal F}$ 
satisfying (\ref{eqn:XuYbot}) to prove Theorem~\ref{thm:R(x,y)}~(2).

\begin{proof}
We may assume $\ell \geq 1$.
Let $\ell'$ be the minimum index such that 
$({\cal C}_{\ell'},{\cal C}_{\ell'+1},\ldots,{\cal C}_{\ell})$ 
is $A$-orthogonal.
If $\ell' = 0$, then the gallery is $A$-orthogonal as required.
Suppose to the contrary that $\ell' > 0$.
We may assume that 
${\cal C}_{\ell'-1}$ and ${\cal C}_{\ell'}$ are ${\cal L}$-adjacent 
and are not $A$-orthogonal.
Let ${\cal X} = \{ X_i \}_{i=0}^{n+m} := {\cal X}_{{\cal C}_{\ell'-1}}$, 
${\cal X}' := {\cal X}_{{\cal C}_{\ell'}}$ 
and  ${\cal Y} := {\cal Y}_{{\cal C}_{\ell'-1}} = {\cal Y}_{{\cal C}_{\ell'}}$.
For some $j$, we have 
${\cal X}' = {\cal X} \setminus \{X_j\} \cup \{X_j'\}$, where
$X_j' \neq g(X_{j-1},X_{j+1}, {\cal Y}^{\bot})$ (or $g(X_{j-1},X_{j+1}, {\cal Y}^{\bot})$ is not defined).
Consider an $A$-orthogonal frame $\langle e_1,e_2,\ldots,e_m,f_1,f_2,\ldots,f_n \rangle$
containing $\langle {\cal X}' \cup {\cal Y}^{\bot} \rangle \times \langle {{\cal X}'}^{\bot} \cup {\cal Y} \rangle$.

For some $t \in (0,1)$,	
the point $(x^t,y^t) = (1-t) (x^0,y^0) + t (x^*,y^*)$  
belongs to the intersection of 
maximal simplices corresponding to ${\cal C}_{\ell'-1}$ and ${\cal C}_{\ell'}$.
By Lemma~\ref{lem:A-ortho_gallery}, the frame $\langle e_1,e_2,\ldots,e_m,f_1,f_2,\ldots,f_n \rangle$ contains ${\cal C}_{\ell}$. 
Regard $\langle e_1,e_2,\ldots,e_m \rangle$ as $2^{\{1,2,\ldots,m\}}$.
Then $X_{j+1} \setminus X_{j-1} = \{a,b\}$, 
$X_j' = \{a\} \cup X_{j-1}$, and $\tilde X_j := \{b\} \cup X_{j-1}$ 
for distinct elements $a,b \in \{1,2,\ldots,m\}$.
Also $K(\langle e_1,e_2,\ldots,e_m,f_1,f_2,\ldots,f_n \rangle) \simeq [0,1]^m \times [0,1]^n$
contains both $(x^t,y^t)$ and $(x^*,y^*)$.
Now $[(x^t,y^t),(x^*,y^*)]$ is the segment in $[0,1]^m \times [0,1]^n$

Consider $x^t$ and $x^*$ 
in the $\langle e_1,e_2,\ldots,e_m \rangle$-coordinate, 
In the original coordinate $x^t = \sum_i \lambda_i X_i$, 
the coefficient $\lambda_j$ of $X_j$ is zero.
Thus $x^t_a = x^t_b$ holds. 
In $x^{t+ \epsilon}$ for small $\epsilon > 0$,
the coefficient of $X'_i$ becomes positive.
This means that $x^{t+ \epsilon}_a > x^{t+ \epsilon}_b$.
Consequently $x_a^* > x_b^*$ holds.		
Let $\tilde x$ be obtained from $x^*$ 
by interchanging the $a$-th and $b$-th coordinates of $x^*$.		
By $x^t_a = x^t_b$, it holds
\begin{equation*}
d(x^t,x^*) = d(x^t,\tilde x).
\end{equation*}
By $d(x^0,\tilde x) \leq d(x^0,x^t) + d(x^t,\tilde x) = d(x^0,x^t) + d(x^t,x^*) = d(x^0, x^*)$,
we have
\begin{equation}\label{eqn:d(x')}
d(x^0,\tilde x) \leq d(x^0, x^*).
\end{equation}

\noindent	
{\bf Case 1:} $g(X_{j-1},X_{j+1},{\cal Y}^{\bot})$ is defined.
We are going to show 
\begin{equation}\label{eqn:contra}
\overline{R}(\tilde x, y^*) + \frac{1}{2\lambda} ( d(x^0,\tilde x)^2 + d(y^0,y^*)^2 ) 
\leq \overline{R}(x^*, y^*) + \frac{1}{2\lambda} ( d(x^0,x^*)^2 + d(y^0,y^*)^2 ), 
\end{equation}
which is a contradiction to its unique optimality of $(x^*,y^*)$.
Notice that $\tilde X_i = g(X_{j-1},X_{j+1}, {\cal Y}^{\bot})$ 
also belongs to $\langle e_1,e_2,\ldots,e_m \rangle$
(since it is generated by $X_{j-1}$, $X_{j+1}$, and ${\cal Y}^{\bot}$).
Since $X_j' = \{a\} \cup X_{j-1} \neq g(X_{j-1},X_{j+1},{\cal Y}^{\bot})$,  
by Lemma~\ref{lem:g_Boolean}, 
chain ${\cal Y}^{\bot}$ contains $b$ before $a$. 
Then ${\cal Y}^{\bot \bot}$ contains $b$ before $a$.
This must be $a \in \{1,2,\ldots,r\}$.
Consider $y^t$ 
in the $\langle f_1,f_2,\ldots,f_n \rangle$-coordinate.

\noindent
{\bf Case 1-1}: $b \in \{1,2,\ldots,r\}$.  
In this case, ${\cal Y}$ also contains $b$ before $a$,
since ${\cal Y}$ is obtained from ${\cal Y}^{\bot \bot}$ 
by adding elements $r+1,r+2,\ldots,n$ (see the proof of Theorem~\ref{thm:R(x,y)}~(1)).
Thus 
\begin{equation*}
y^t_b > y^t_a.
\end{equation*}

\noindent
{\bf Case 1-1-1}: $y_b^* \geq y_a^*$. 
Recall Theorem~\ref{thm:Lovasz} that $\overline R(x^*,y^*)$ is given by
\[
\overline{R}(x^*,y^*) = \sum_{i=1}^r \max \{0, x_i^* - y_i^*\}.
\]
By $x_a^* > x^*_b$ and $y_b^* \geq y_a^*$, it is easy to verify 
\begin{eqnarray*}
\max\{0, x_a^* - y_a^*\} + \max\{0, x_b^* - y_b^*\} 
&\geq &\max\{0, x_b^* - y_a^*\} + \max\{0, x_a^* - y_b^*\} \\
& =& 
\max\{0, \tilde x_a - y_a^*\} + \max\{0, \tilde x_b - y_b^*\}.
\end{eqnarray*}
For example, 
if $x_a^* \geq y_b^* \geq x_b^* \geq y_a^*$, 
the LHS is $x_a^* - y_a^*$ and the RHS is 
$(x_b^* - y_a^*) + (x_a^*- y_b^*)  \leq x_a^* - y_a^*$. 
Thus we obtain
\begin{equation*}
\overline{R}(\tilde x, y^*) \leq \overline{R}(x^*,y^*), 
\end{equation*}
By (\ref{eqn:d(x')}), we have contradiction (\ref{eqn:contra}).

\noindent
{\bf Case 1-1-2}: $y_b^* < y_a^*$.

Let $\tilde y$ be obtained from $y^*$ 
by interchanging the $a$-th and $b$-th coordinates of $y^*$.
Clearly
\begin{equation*}
\overline{R}(\tilde x, \tilde y) = \overline{R}(x^*,y^*).
\end{equation*}
Since $y^t_b > y^t_a$ and $y_b^* < y_a^*$, 
it must hold $y^{t'}_a = y^{t'}_b$ for some $t' > t$, 
and hence $d(y^{t'},y^*) = d(y^{t'},\tilde y)$.
Thus we have
\begin{equation*}
 d(y^0,\tilde y) \leq  d(y^0,y^{t'}) + d(y^{t'},\tilde y) = d(y^0,y^{t'}) + d(y^{t'},y^*) 
= d(y^0, y^*)   
\end{equation*}
Then we obtain a contradiction:
\begin{equation*}
\overline{R}(\tilde x, \tilde y) + \frac{1}{2\lambda} ( d(x^0,\tilde x)^2 + d(y^0,\tilde y)^2 ) 
\leq \overline{R}(x^*, y^*) + \frac{1}{2\lambda} ( d(x^0,x^*)^2 + d(y^0,y^*)^2 ). 
\end{equation*}

\noindent
{\bf Case 1-2}: $b \in V^{\bot}$ i,e., $b > r$.
By $\max \{0, x^*_a - y^*_a \} \geq \max \{0, x^*_b - y^*_a\} = 
\max \{0, \tilde x_a - y^*_a\}$, and $\overline{R}(x^*,y^*) \geq \overline{R}(\tilde x, y^*)$. 
we obtain a contradiction (\ref{eqn:contra}).

\noindent
{\bf Case 2:} $g(X_{j-1},X_{j+1}, {\cal Y}^{\bot})$ 
is not defined. 
In this case, both $a$ and $b$ belong to $V^{\bot}$, 
Namely $a,b > r$ holds.
Then $\overline{R}(\tilde x, y^*) = \overline{R}(x^*,y^*)$.
Thus we obtain (\ref{eqn:contra}). 
\end{proof}
Finally we remove the genericity assumption.
\begin{Lem}\label{lem:remove}
	For $z^0 = (x^0,y^0) \in K({\cal L} \times \check{\cal M})$, and a maximal simplex $K$ containing $z^0$,
	there is $z \in K$ 
	such that 
	\begin{itemize}
		\item[{\rm (1)}] $[z,J^{\overline{R}}_{\lambda}(z)]$ is generic, and
		\item[{\rm (2)}]  $J^{\overline{R}}_{\lambda}(z^0)$ is contained in the simplex containing $J^{\overline{R}}_{\lambda}(z)$ as its relative interior.  
	\end{itemize}
\end{Lem}
Thus we can choose points $x$ and $y$ from any maximal simplices 
corresponding to ${\cal X}$ and ${\cal Y}$ 
such that $[(x,y), J_{\lambda}^{\overline{R}}(x,y)]$ 
is generic, and the simplex $K^*$ containing $J_{\lambda}^{\overline{R}}(x,y)$ in its relative interior contains $(x^*,y^*) = J_{\lambda}^{\overline{R}}(x^0,y^0)$.
Therefore, for any $A$-orthogonal frame ${\cal F}$ 
satisfying (\ref{eqn:XuYbot}), the subcomplex $K({\cal F})$ 
contains  $J_{\lambda}^{\overline{R}}(x,y)$ and $K^* \ni (x^*,y^*)$.
\begin{proof}
	First notice that $J = J_{\lambda}^{\overline{R}}$ 
	is nonexpansive~\cite{Jost95}, and is continuous.
	Let $B(z, \epsilon)$ denote the open ball with center $z$ and radius $\epsilon > 0$.
	For sufficiently small $\epsilon >0$ and every $u \in B(z^0, \epsilon)$, 
	the simplex $K^*$ containing 
	$J(u)$ as its relative interior also contains~$J(z^0)$. 
	Thus, by perturbing $z^0$, 
	we can assume in advance that $z^0$ belongs to the interior of $K$.
	Let $\epsilon > 0$ be sufficiently small so that $B(z^0,\epsilon) \subseteq K$.
	We can replace $z^0$ by a point $z'$ in 
 $B(z^0, \epsilon)$ that 
	maximizes the dimension of 
  the simplex $K^*$ containing $J(z')$ as its relative interior.
	Then we can assume that for sufficiently small $\epsilon > 0$, 
	the image $J(B(z^0, \epsilon))$ 
   belongs to the relative interior of $K^* (\ni J(z^0))$.

	It suffices to show that there is $z \in B(z^0, \epsilon)$ 
	such that $[z, J(z)]$ is generic.
    Consider a frame ${\cal F}$ containing the supports of $z$ and $K^*$.
    Regard $K({\cal F}) \simeq [0,1]^{n+m}$.
    Consider the affine hull of $K^*$, 
    which is represented by linear equation $A u = b$.
    In $K^*$, Lov\'asz extension $\overline R$ 
    is a linear function $u \mapsto c^{\top} u$.
    For every $u \in B(z^0, \epsilon)$, 
   resolvent $J(u)$ is the unique minimizer of 
    \begin{equation*}
    {\rm Min.}\ c^{\top} v + \frac{1}{2\lambda} \|v - u \|_2^2 \quad  {\rm s.t.} \ Av = b.
    \end{equation*}  
    This is an equality-constrained quadratic program.
    By the Lagrange multiplier method, 
    we obtain an explicit formula of $J$:
    \begin{equation*}
    J(u) =(I - A^{\top} (AA^{\top})^{-1}A) u + c',
    \end{equation*}
    where $c'$ is a constant vector.
    Consider geodesic (segment) $[u,J(u)]$.
    For each $t \in (0,1)$, 
    define $\varphi_t: B(z^0,\epsilon) \to [0,1]^{n+m}$ by
    \[
       \varphi_t (u) := (1-t) u + t J(u)  = (I - t A^{\top} (AA^{\top})^{-1}A) u  + t c'.
    \]
    Here $A^{\top} (AA^{\top})^{-1}A$ is a projection, 
    and its eigenvalue is $0$ or $1$.
    Hence
     $(I - t A^{\top} (AA^{\top})^{-1}A)$ is nonsingular for $t \in (0,1)$.
    This implies that $\varphi_t(B(z^0, \epsilon))$ 
    is an open neighborhood of $\varphi_t(z)$ for $t \in (0,1)$.
    Suppose that open segment $(z^0,J(z^0))$ meets 
    simplices $F_1,F_2,\ldots,F_{\ell}$ of dimension at most $n+m-2$.
    Now $\epsilon$ is small. For every $u \in B(z^0,\epsilon)$, 
    any simplex of dimension at most $n+m-2$ which $(u,J(u))$ can meet
    is one of $F_1,F_2,\ldots,F_{\ell}$.
    For $i \in \{1,2,\ldots,\ell\}$, 
the set of points $u \in B(z^0, \epsilon)$ 
with $\varphi_t(u) \in F_i$ belongs to 
an affine subspace of dimension $n+m-2$.
Consequently, the set of points $u \in B(z^0, \epsilon)$ 
with $\varphi_t(u) \in F_i$ for some $t \in (0,1)$, i.e., 
$(u,J(u))$ meets $F_i$, 
must belong to a hypersurface ${\cal H}_i$ (of dimension $n+m-1$).
Therefore, choose $z$ from 
$B(z^0, \epsilon) \setminus \bigcup_{i=1}^{\ell} {\cal H}_i$.
Then $(z,J(z))$ meets none of simplices $F_1,F_2,\ldots,F_{\ell}$. 
Namely $[z, J(z)]$ is generic, as required.	
\end{proof}

\section{Block-triangularization of partitioned matrix}\label{sec:DM}
In this section, we present implications of Theorem~\ref{thm:main'} 
on a block-triangularization of a partitioned matrix.
\subsection{DM-decomposition}
Let $A = (A_{\alpha \beta})$ be a partitioned matrix as above.
Consider MVSP for $A$. 
A vanishing subspace $(X_1,X_2,\ldots,X_\mu, Y_1,Y_2,\ldots,Y_\nu)$ 
is simply denoted by $(X,Y)$, 
where $X$ and $Y$ denote tuples of subspaces $X_{\alpha}$ and $Y_{\beta}$, respectively.
We say that $(X,Y)$ 
is a vanishing subspace with dimension
$\dim X + \dim Y$, where 
$\dim X: = \sum_{\alpha} \dim X_{\alpha}$ and $\dim Y := \sum_{\beta} \dim Y_{\beta}$.
Formally speaking, $(X,Y)$ represents subspace 
$\bigoplus_{\alpha} X_\alpha \times  \bigoplus_{\beta} Y_\beta$ 
of $\bigoplus_{\alpha} {\bf F}^{m_{\alpha}} \times \bigoplus_{\beta} {\bf F}^{n_{\beta}}$   
on which the bilinear form $\bigoplus_{\alpha} {\bf F}^{m_{\alpha}} \times \bigoplus_{\beta} {\bf F}^{n_{\beta}} \to {\bf F}$ defined by
\[
(u,v) \mapsto \sum_{\alpha, \beta} u_{\alpha}^{\top}A_{\alpha \beta} v_{\beta}
\]
vanishes, where $u_{\alpha}$ (resp. $v_{\beta}$) 
is the natural projection of $u$ to ${\bf F}^{m_{\alpha}}$ (resp. ${\bf F}^{n_{\beta}}$).
A vanishing subspace of a maximum dimension is 
called a maximum vanishing subspace, 
abbreviated as an {\em mv-subspace}.

Let ${\cal S} = {\cal S}_{A}$ denote the modular lattice of all mv-subspaces for $A$, 
where the partial order is given by $(X,Y) \preceq (X',Y')$ if and only if 
$X_{\alpha} \subseteq X_{\alpha}'$ and $Y_{\beta} \supseteq Y_{\beta}'$ for each $\alpha,\beta$.
Consider a chain $(X^0,Y^0) \prec (X^1,Y^1) \prec \cdots \prec (X^{\ell},Y^{\ell})$ 
of mv-subspaces.
For each $\alpha$, choose a base $E_{\alpha} = \{e^{\alpha}_1,e^{\alpha}_2,\ldots, e^{\alpha}_{m_{\alpha}}\}$ of ${\bf F}^{m_{\alpha}}$
such that $E_\alpha^k = \{e^{\alpha}_1,e^{\alpha}_2,\ldots, e^{\alpha}_{k_{\alpha}}\}$, 
for some $k_{\alpha}$, is a base of $X_{\alpha}^{k}$ for $k=1,2,\ldots,\ell$.
For each $\beta$, choose a base $F_{\beta} = \{f^{\beta}_1,f^{\beta}_2,\ldots, f^{\beta}_{n_{\beta}}\}$ of ${\bf F}^{n_{\beta}}$
such that $F_{\beta}^{k} = \{f^{\beta}_1,f^{\beta}_2,\ldots, f^{\beta}_{k_{\beta}}\}$, 
for some $k_{\beta}$, is a base of $Y_{\beta}^{k}$ for $k=1,2,\ldots,\ell$.
Then $\bigcup_{\alpha} E_{\alpha}$ is regarded as a base of ${\bf F}^m$
via canonical injection ${\bf F}^{m_{\alpha}} \hookrightarrow \bigoplus_{\alpha} {\bf F}^{m_{\alpha}}$, 
and $\bigcup_{\beta} F_{\beta}$ is regarded as a base of ${\bf F}^n$ similarly.
Also $\bigcup_{\alpha} E_{\alpha}^k$ is a base of $X^k$, 
and $\bigcup_{\beta} F_{\beta}^k$ is a base of $Y^k$.
Then the change of the bases gives rise to a transformation of the form (\ref{eqn:trans}). 
By rearranging rows and columns, 
we obtain the following block-triangular form:
\begin{equation}\label{eqn:block_triangularization}
\left( 
\begin{array}{ccccc}
D_{ \ell+1} & & & & \\
O & D_{\ell} & & & \\ 
O & O  & \ddots & &  \\
\vdots & \vdots & \ddots  & \ D_{1} & \\[0.7em]
O      & O      & \cdots & O & D_{0}  
\end{array}
\right),
\end{equation}
where the diagonal block $D_{k}$ is a square matrix 
of size $\dim X^{k+1} - \dim X^{k} = \dim Y^{k} - Y^{k+1}$
for $k=1,2,\ldots,\ell$, 
$D_0$ is a matrix of $\dim X_0$ rows and $n - \dim Y_0 (< \dim X_0)$ columns  
and $D_{\ell +1}$ is a matrix of $m - \dim X_{\ell} (< \dim Y_{\ell})$ rows and  
$\dim Y_{\ell}$ columns.

For any vanishing subspace $(X,Y)$ of $A$, recall the introduction 
that the following inequality holds:
\begin{equation*}
\dim X + \dim Y \leq m + n - \rank A.
\end{equation*} 
In particular, $m + n - \rank A$ is an upper bound 
of the maximum vanishing dimension, though it is not attained in general. 
Ito, Iwata, and Murota~\cite{ItoIwataMurota94} 
mainly focus the case where this bound is attained.
In this case, the resulting block-triangular 
form~(\ref{eqn:block_triangularization}) 
satisfies the rank-condition that each $D_k$ is of row- or column-full rank.
Such a block triangular form is particularly called {\em proper}.
We here do not impose the properness on decomposition~(\ref{eqn:block_triangularization}).

The {\em DM-decomposition} of $A$ is 
the most refined block triangularization such that 
the chain of mv-subspaces is taken to be maximal in ${\cal S}$.
The original DM-decomposition~\cite{DulmageMendelsohn58} 
corresponds to the case of $m_{\alpha} = n_{\beta} = 1$ for all $\alpha,\beta$. 
The {\em combinatorial canonical form (CCF)} 
for a multilayered mixed matrix~\cite{MurotaIriNakamura87}  
corresponds to the case of $n_{\beta} = 1$ for all $\beta$.
There are polynomial time algorithms (based on bipartite matching and matroid union) 
to obtain DM-decompositions for these cases, 
whereas no polynomial time algorithm is known for the general case.

MVSP asks for one mv-subspace.
On the other hand, 
the DM-decomposition needs
a maximal chain of mv-subspaces.
Therefore, solving MVSP
is not enough to 
obtaining the DM-decomposition.

\paragraph{On the difficulty of DM-decomposition.}
Obtaining the DM-decomposition 
cannot avoid issues of numerical analysis/computation and 
the algebraically closedness of base field ${\bf F}$. 
Consider the following partitioned matrix of type $(n,n;n,n)$
\begin{equation}\label{eqn:2x2}
\left(\begin{array} {cc}
A & B \\
C & D
\end{array} 
\right),
\end{equation}
where $A,B,C,D$ are all nonsingular. 
Finding the DM-decomposition of this matrix reduces to the eigenvalue problem as follows.
Suppose that 
$(X_1, X_2, Y_1, Y_2)$ 
is a vanishing subspace. 
By the nonsingularity of the submatrices, it must hold that  
$\dim X_{\alpha} + \dim Y_{\beta} \leq n$. 
Consequently, trivial vanishing subspaces $(\{0\},\{0\}, {\bf F}^n, {\bf F}^n)$ and 
$({\bf F}^n, {\bf F}^n, \{0\},\{0\})$  
are maximum with dimension $2n$. 
Suppose that 
$(X_1,X_2, Y_1,Y_2)$ is an mv-subspace.
Then it must hold $\dim X_{1} = \dim X_{2} = n - \dim Y_1 = n - \dim Y_2$.
Moreover, from
$A Y_1 = (X_1)^{\bot} = B Y_2$ and $C Y_1 = (X_1)^{\bot} = D Y_2$, we obtain
\[
(C^{-1}DB^{-1}A) Y_1 = Y_1,
\] 
where $(\cdot)^{\bot}$ means the orthogonal subspace with respect to 
the standard inner product.
If such $Y_1$ is given, 
then we can recover mv-subspace 
$(X_1, X_2, Y_1, Y_2)$. 
This implies that finding a maximal chain of mv-subspaces
is equivalent to finding a maximal chain of invariant subspaces 
of matrix $C^{-1}DB^{-1}A$.
In the case where the base field ${\bf F}$ is algebraically-closed,
the Schur decomposition finds such a chain of invariant subspaces and
triangularizes
$C^{-1}DB^{-1}A$ by a similarity transformation,    
where the resulting triangular form has all eigenvalues in diagonals.
Consequently, we obtain a maximal chain of mv-subspaces and 
the DM-decomposition with four diagonal blocks of size $2 \times 2$. 
In particular, the DM-decomposition may change 
when ${\bf F}$ is not algebraically-closed and
the matrix is considered in an extension field of ${\bf F}$.
A simple example of such a matrix (over ${\bf Q}$) 
is given in~\cite[6.2]{IwataMurota95}
 
A more difficult situation occurs.
Consider the following partitioned matrix of type 
$(n,n,n;n,n,n)$
\begin{equation}\label{eqn:3x3}
\left(\begin{array} {ccc}
A & B &  B' \\
C & D & D' \\
\tilde C & \tilde D & E
\end{array}\right),
\end{equation}
where all submatrices are nonsingular.
By the same argument, the maximum vanishing dimension is $3n$. 
Also, if $(X_1,X_2,X_3, Y_1,Y_2,Y_3)$ 
is an mv-subspace, then $Y_1$ must satisfy
\[
(C^{-1}DB^{-1}A) Y_1 = Y_1,\quad  
({\tilde C}^{-1}{\tilde D}{B}^{-1}A) Y_1 = Y_1, \quad 
({\tilde C}^{-1} {\tilde D}{D}^{-1}C) Y_1 = Y_1
\] 
Namely $Y_1$ is a common invariant subspace of three matrices.
Therefore the problem of finding the DM-decomposition
includes the {\em common invariant subspace problem}.
This extremely difficult problem undergoes current research in numerical analysis/computation 
(see e.g., \cite{ArapuraPeterson04,JamiolkowskiaPastuszakb15}), and 
a satisfactory algorithm is not yet obtained (as far as we recognize).

\subsection{Quasi DM-decomposition}

Here we introduce the concept of {\em quasi DM-decomposition}, which is
a block-triangular form {\em coarser} than the DM-decomposition 
but does not depend on base field ${\bf F}$ and still generalizes
important special cases (the original DM-decomposition and CCF).
It turns out that a quasi DM-decomposition 
corresponds exactly to a chain of mv-subspaces 
{\em detectable} by solving WMVSP, and is obtained in polynomial time.
We believe that obtaining a quasi DM-decomposition is a limit 
which we can do by combinatorial or optimization methods.

Let $A = (A_{\alpha \beta})$ be a partitioned matrix as above, and ${\cal S}$ 
the lattice of all mv-subspaces for $A$. 
A vanishing space $(X,Y)$ is said to be {\em trivial} 
if $X_{\alpha} = 0$ for each $\alpha$ or $Y_{\beta} = 0$ for each $\beta$.
Other vanishing spaces are said to be {\em nontrivial}.
$A$ is called {\em DM-irreducible} if ${\cal S}$ consists only of trivial mv-subspaces, 
and called {\em DM-regular} if ${\cal S}$ 
contains both of the trivial mv-subspaces, or equivalently, 
if the maximum vanishing dimension is equal to $n$ and $m$. 
In particular, a DM-regular matrix is necessarily a square matrix.
In the DM-decomposition~(\ref{eqn:block_triangularization}), 
each diagonal block $D_{k}$ is DM-irreducible, 
and is DM-regular if $1 \leq k \leq \ell$.

To formulate quasi DM-decomposition, 
we introduce the notion of the quasi DM-irreducibility.
Partitioned matrix $A$ is called {\em quasi DM-irreducible}
if for each nontrivial mv-subspace $(X,Y) \in {\cal S}$ 
there are positive integers $k,\ell$ with $k < \ell$
such that for all $\alpha,\beta$ it holds
\begin{equation}\label{eqn:quasi}
\frac{\dim X_{\alpha}}{m_{\alpha}} 
= \frac{n_{\beta} - \dim Y_{\beta}}{n_{\beta}} = \frac{k}{\ell}.
\end{equation}
This means that any nontrivial mv-subspace of a quasi DM-irreducible matrix has 
a common ratio of dimensions in ${\bf F}^{m_{\alpha}} \times {\bf F}^{n_{\beta}}$ 
for all $\alpha, \beta$.
Obviously the quasi DM-irreducibility is a weaker notion than the DM-irreducibility.
If $A$ is quasi DM-irreducible and admits a nontrivial mv-subspace $(X,Y)$, 
then  $\max(m,n) \leq \sum_{\alpha} \dim X_{\alpha} + \sum_{\beta} \dim Y_{\beta} 
= (k/\ell) \sum_{\alpha} m_{\alpha} + ((\ell-k)/\ell) \sum_{\beta} n_{\beta} 
= (k/ \ell) m + ((\ell-k)/\ell)n \leq \max (m,n)$, and
necessarily the maximum vanishing dimension is equal to $n= m$, 
which implies that $A$ is DM-regular.
In particular, the DM-irreducibility and quasi DM-irreducibility are the same 
for a non-square partitioned matrix.

For $n \geq 2$, 
any $n \times n$ nonsingular matrix $A$, 
viewed as a partition matrix of type $(n;n)$, 
is not DM-irreducible but quasi DM-irreducible.
Indeed, for any proper nonzero subspace $X$, 
$(X, (XA)^\bot)$ is a nontrivial mv-subspace 
with $(n - \dim (XA)^{\bot})/n = (n - (n- \dim X))/n = \dim X/n$.
Also, a partitioned matrix of form (\ref{eqn:2x2})
is quasi DM-irreducible and not DM-irreducible if ${\bf F}$ is algebraically closed.
More generally, 
any partitioned matrix of consisting $n \times n$ nonsingular 
submatrices, such as (\ref{eqn:3x3}), is quasi DM-irreducible.

A {\em quasi DM-decomposition} of $A$ is
a block-triangular form (\ref{eqn:block_triangularization}) 
such that each diagonal block is quasi DM-irreducible.
The quasi DM-decomposition still generalizes 
an important special case of CCF  ($n_{\beta} = 1$ for all $\beta$).
This fact follows from: 
\begin{Lem}
	Suppose that $A$ is DM-regular with $\gcd (m_1,m_2,\ldots,m_{\mu}, n_1,n_2,\ldots,n_{\nu}) = 1$. 
	Then $A$ is DM-irreducible if and only if $A$ is quasi DM-irreducible.
\end{Lem}
\begin{proof}
	If $A$ admits a nontrivial mv-subspace as in (\ref{eqn:quasi}), 
	then $\ell$ becomes a common divisor of $m_{\alpha},n_\beta$, which is greater than $1$.
\end{proof}

The main result of this section is the following.
\begin{Thm}\label{thm:q-DM}
	A quasi DM-decomposition of a partitioned matrix over ${\bf F}$ 
	can be obtained in polynomial time, provided arithmetic operations on ${\bf F}$ can be done in constant time.
\end{Thm}

The rest of this section is devoted to the proof of this theorem.
The algorithm is based on a simple recursive idea:
Find a nontrivial mv-subspace for $A$ by solving WMVSP with special weights.
If a nontrivial mv-subspaces $(X,Y)$ 
is found, then decompose $A$ into two matrices $A^{X,Y^c}$ and $A^{X^c,Y}$, 
and recurse into $A^{X,Y^c}$ and into $A^{X^c,Y}$.

Now suppose that we are given one (nontrivial) mv-subspace $(X,Y)$.
The two partitioned matrices $A^{X,Y^c}$ and $A^{X^c,Y}$ are constructed as follows. 
For each $\alpha$, choose
a complement $U_{\alpha}$ of $X_{\alpha}$. 
For each $\beta$, choose
a complement $V_{\beta}$ of $Y_{\beta}$. 
Let $A^{X,Y^c}_{\alpha \beta}$ be the matrix representation of
the restriction of $A_{\alpha \beta}$ to $X_{\alpha} \times V_{\beta}$.
Let $A^{X, Y^c} := (A^{X,Y^c}_{\alpha \beta})$ be the partitioned matrix 
consisting of the nonempty matrices among them.
Define $A^{X^c, Y} := (A^{X^c,Y}_{\alpha \beta})$ similarly.
\begin{Lem}\label{lem:recursive}
	Let $(X,Y)$ be an mv-subspace for $A$, and let $(X',Y')$ be 
	a vanishing subspace for $A$ such that $(X',Y') \preceq (X,Y)$.
	The following conditions are equivalent:
	\begin{itemize}
		\item[{\rm (1)}] $(X',Y')$ is an mv-subspace for $A$.
		\item[{\rm (2)}] $(X',Y')$ is represented as $(X',Y+Q)$ with 
		an mv-subspace $(X',Q)$ for $A^{X,Y^c}$. 
	\end{itemize}
\end{Lem}
\begin{proof}
	Let $Q_{\beta} := V_{\beta} \cap Y'_{\beta}$ for $\beta$. 
	Then $Y'_{\beta} = Q_{\beta} + Y_{\beta}$.
	$A_{\alpha \beta} (X'_{\alpha}, Y'_{\beta}) 
	= A_{\alpha \beta}(X'_{\alpha},Y_{\beta}) + A_{\alpha \beta}(X'_{\alpha},Q_{\beta}) = A_{\alpha \beta}^{X,Y^c}(X'_{\alpha},Q_{\beta})$ (since 
	$A_{\alpha \beta}(X'_{\alpha},Y_{\beta}) \subseteq A_{\alpha \beta}(X_{\alpha},Y_{\beta}) = \{0\}$). 
	Thus $A_{\alpha \beta} (X'_{\alpha}, Y'_{\beta}) = \{0\}$ 
	if and only if $A_{\alpha \beta}^{X,Y^c}(X'_{\alpha},Q_{\beta}) = \{0\}$.
	The claim follows from this fact and $\dim Y'_{\beta} = \dim Q_{\beta} + \dim Y_{\beta}$.
\end{proof}

Next we consider to find a nontrivial mv-subspace by solving WMVSP. 
An mv-subspace $(X,Y)$ is called {\em extremal} if 
$(X,Y)$ is 
the unique optimal solution of WMVSP 
for some weights $C_{\alpha},D_{\beta}$.

The minimal and maximal mv-subspaces are extremal.
\begin{Lem}\label{lem:minimal_maximal}
	\begin{itemize}
		\item[{\rm (1)}] Define weights $C_{\alpha},D_{\beta}$ by
		\begin{equation}\label{eqn:maximal}
		C_{\alpha} := m+2 \quad (\alpha=1,2,\ldots,\mu)\quad 
		D_{\beta} := m+1 \quad (\beta = 1,2,\ldots,\nu).
		\end{equation}
		Then an optimal solution of WMVSP is unique, 
		and is equal to the maximal mv-subspace.
		\item[{\rm (2)}] Define weights $C_{\alpha},D_{\beta}$ by
		\begin{equation}\label{eqn:minimal}
		C_{\alpha} := n+1 \quad (\alpha=1,2,\ldots,\mu)\quad 
		D_{\beta} := n+2 \quad (\beta = 1,2,\ldots,\nu).
		\end{equation}
		Then an optimal solution of WMVSP is unique, and is equal to the minimal mv-subspace.
	\end{itemize}
\end{Lem}
\begin{proof}
	It suffices to prove (1). 
	Let $(X,Y)$ be the unique maximal mv-subspace, and let $(X',Y')$ be an arbitrary vanishing subspace.
	Then 
	\begin{eqnarray}
	&& \sum_{\alpha} C_{\alpha}\dim X_{\alpha} + \sum_{\beta} D_{\beta} \dim Y_{\beta} - \sum_{\alpha} C_{\alpha}\dim X'_{\alpha} - \sum_{\beta} D_{\beta}\dim Y'_{\beta} \nonumber \\ 
	&& = (m+1) (\dim X + \dim Y - \dim X' - \dim Y') + \dim X - \dim X' \label{eqn:XX'}  
	\end{eqnarray}
	If $(X',Y')$ is not an mv-subspace, then (\ref{eqn:XX'}) $\geq m+1 - m > 0$.
	If $(X',Y')$ is a nonmaximal mv-subspace, then $\dim X > \dim X'$, and (\ref{eqn:XX'}) $> 0$.
	Thus $(X,Y)$ is the unique optimal solution of WMVSP.
\end{proof}
Therefore we may focus on a DM-regular partitioned matrix. 
\begin{Lem}\label{lem:CD}
	Suppose that $A$ is DM-regular.
	\begin{itemize}
		\item[{\rm (1)}] For $\alpha' \in \{1,2,\ldots,\mu \}$, define weights $C_{\alpha},D_{\beta}$ by 
		\begin{eqnarray}
		C_{\alpha} & := & \left\{ \begin{array}{ll}   
		m(2m_{\alpha'} + 1) (2m_{\alpha'} + 2) & {\rm if}\ \alpha = \alpha',  \\
		m(2m_{\alpha'} + 1)^2 & {\rm otherwise},
		\end{array}
		\right. (\alpha = 1,2,\ldots, \mu), \nonumber \\
		D_{\beta} & := & (2m_{\alpha'} + 1)(m(2m_{\alpha'} + 1) + m_{\alpha'})  \quad \quad (\beta = 1,2,\ldots, \nu). \label{eqn:CD1}
		\end{eqnarray}	
		Then any optimal solution of WMVSP is an mv-subspace.
		\item[{\rm (2)}] For $\beta' \in \{1,2,\ldots,\nu \}$, define weights $C_{\alpha},D_{\beta}$ by 
			\begin{eqnarray}
			C_{\alpha} & := & 
			(2n_{\beta'} + 1)(n(2n_{\beta'} + 1) + n_{\beta'}) \quad \quad (\alpha =1,2,\ldots,\mu),  \nonumber \\
			D_{\beta} & := & \left\{ \begin{array}{ll}   
			n(2n_{\beta'} + 1) (2n_{\beta'} + 2) & {\rm if}\ \beta = \beta', \\
			n(2n_{\beta'} + 1)^2 & {\rm otherwise},
			\end{array} 
			\right. \quad (\beta = 1,2,\ldots,\nu). \label{eqn:CD2}
			\end{eqnarray}	
		Then any optimal solution of WMVSP is an mv-subspace.
	\end{itemize}
\end{Lem}
\begin{proof}
	It suffices to prove (1). 
	Let $(X,Y)$ be an optimal solution of WMVSP, and 
	let $(X',Y')$ be a vanishing subspace.
	Then, letting $M := 2m_{\alpha'} + 1$, we have
	\begin{eqnarray*}
   && m M^2 \dim X + m M \dim X_{\alpha} + M ( m M + m_{\alpha'}) \dim Y \\
   && \geq m M^2 \dim X' + m M \dim X'_{\alpha} + M ( m M + m_{\alpha'}) \dim Y'. 
	\end{eqnarray*}
	From this, we have
	\begin{eqnarray*}
	\dim X + \dim Y &\geq& \dim X' + \dim Y' + \frac{\dim X'_{\alpha} - \dim X_{\alpha}}{M} 
	+ \frac{m_{\alpha'}(\dim Y' - \dim Y)}{m M} \\
	&  \geq & \dim X' + \dim Y' - \frac{2 m_{\alpha'}}{2 m_{\alpha'}+1} >  \dim X' + \dim Y' - 1, 
	\end{eqnarray*}
	where we use $\dim Y' \leq n = m$.
	This implies that $\dim X + \dim Y \geq \dim X' + \dim Y'$. 
	Thus $(X,Y)$ is an mv-subspace.
\end{proof}

\begin{Thm}\label{thm:q-DM_irreducible}
	Suppose that $A$ is DM-regular.
	The following conditions are equivalent:
	\begin{itemize}
		\item[{\rm (1)}] $A$ is quasi DM-irreducible.
		\item[{\rm (2)}] There is no extremal nontrivial mv-subspace. 
		\item[{\rm (3)}] For each $\alpha' \in \{1,2,\ldots,\mu\}$, 
	the trivial mv-subspaces are optimal to WMVSP with weights~$(\ref{eqn:CD1})$,  
	and, for each $\beta' \in \{1,2,\ldots,\nu\}$, the trivial mv-subspaces are optimal to 
	WMVSP with weights~$(\ref{eqn:CD2})$. 
	\end{itemize}
\end{Thm}
\begin{proof}
	(1) $\Rightarrow$ (2). Let $(X,Y)$ be 
	a nontrivial mv-subspace (of dimension $n = m$). 
	Then there are positive integers $k, \ell$ satisfying (\ref{eqn:quasi}).
	For any weights $C_{\alpha},D_{\beta}$, we have
	\[
	\sum_{\alpha} C_{\alpha} \dim X_{\alpha} + \sum_{\beta} D_{\beta} \dim Y_{\beta}
	=  \frac{\ell}{k} \sum_{\alpha} C_{\alpha} m_{\alpha} + \frac{k- \ell}{k}\sum_{\beta} D_{\beta} n_{\beta}.
	\]
	Here $\sum_{\alpha} C_{\alpha} m_{\alpha}$ and $\sum_{\beta} D_{\beta} n_{\beta}$ are 
	the weights of two trivial mv-subspaces.
	This means that $(X,Y)$ is never a unique optimal solution of WMVSP.
	
	(2) $\Rightarrow$ (3). 
	Let $(X,Y)$ be an optimal solution of WMVSP under weights $(\ref{eqn:CD1})$. 
	By Lemma~\ref{lem:CD}, the space $(X,Y)$ is an mv-subspace.
	If $(X,Y)$ has the weight greater than the weight of the trivial mv-subspace, then $(X,Y)$ is nontrivial, and this implies the existence of an extremal mv-subspace other than the trivial ones. 
	
	(3) $\Rightarrow$ (1). 
Suppose that $A$ is not quasi DM-irreducible.
There is a nontrivial mv-subspace $(X,Y)$ such that one of the following holds:
\begin{itemize}
\item[(i)] $\dim X_{\alpha}/m_{\alpha} \neq \dim X_{\alpha'}/m_{\alpha'}$ for some $\alpha,\alpha'$.
\item[(ii)]  $\dim Y_{\beta}/n_{\beta} \neq \dim Y_{\beta'}/n_{\beta'}$ for some $\beta,\beta'$.
\item[(iii)] $\dim X_{\alpha}/m_{\alpha} \neq (n_{\beta} - \dim Y_{\beta})/n_{\beta}$ for some $\alpha,\beta$.
\end{itemize}
We may assume that (i) or (ii) holds.
Indeed, suppose that (iii) holds and both (i) and (ii) do not hold.
There are some positive integers $k,\ell,k',\ell'$ with 
$k/\ell \neq k'/\ell'$ such that  $\dim X_{\alpha}/m_{\alpha} = k/{\ell}$ and $\dim Y_{\beta}/n_{\beta} = (\ell' - k')/\ell'$ hold for all $\alpha, \beta$. Thus we have 
\[
\sum_{\alpha} \dim X_{\alpha} + \sum_{\beta} \dim Y_{\beta} = \frac{k}{\ell} \sum_{\alpha} m_{\alpha} +  \frac{\ell' - k'}{\ell'} \sum_{\beta} n_{\beta} = \left( \frac{k}{\ell} + \frac{\ell' - k'}{\ell'} \right) n \neq n.
\]
This is a contradiction since the maximum vanishing dimension is $n$.

We may assume that (i) holds.
Let $\kappa_{\alpha} := \dim X_{\alpha}/m_{\alpha}$ and 
$\kappa := \dim X/m = (n- \dim Y)/n$.
Let $\alpha'$ denote an index $\alpha$ having 
the maximum $\kappa_{\alpha}$.
Then we have
\[
\kappa = \frac{\sum_{\alpha} \kappa_\alpha m_{\alpha}}{m} < \frac{\kappa_{\alpha'} \sum_{\alpha} m_{\alpha}}{m} = \kappa_{\alpha'}.
\]
Consider the optimal value of WMVSP with weights (\ref{eqn:CD1}) 
for index $\alpha'$, which is given by
\begin{eqnarray*}
&& Mm(M \dim X + \dim X_{\alpha'}) + M(m M + m_{\alpha'}) \dim Y \\
&& = M^2m (\dim X + \dim Y) + M m \dim X_{\alpha'} + M m_{\alpha'} \dim Y \\
&& = M^2 m^2 + M (m m_{\alpha'} \kappa_{\alpha'} + m_{\alpha'} n (1 - \kappa)) \\
&& = M^2 m^2 + M m m_{\alpha'} (\kappa_{\alpha'} - \kappa+ 1) > M^2 m^2 + M m m_{\alpha'}, 
\end{eqnarray*}
where we let $M := 2 m_{\alpha'} +1$, and we use $n=m$ and $\kappa_{\alpha'} > \kappa$.
Here $M^2 m^2 + Mmm_{\alpha'}$ is the weight of the trivial ones.
In particular, the trivial vanishing spaces are not optimal.	
\end{proof}
Now we are ready to describe an algorithm to obtain a quasi DM-decomposition,
The algorithm outputs a chain of mv-subspaces corresponding to a quasi DM-decomposition,
 which we call a {\em q-DM chain}.
\begin{description}
	\item[Algorithm: q-DM]
	\item [Input:] A partitioned matrix $A$.
	\item [Output:] A q-DM chain ${\cal C}$ of mv-subspaces for $A$.
	\item [1:] Solve WMVSP for $A$ under weights (\ref{eqn:maximal}) 
	to obtain the maximal mv-subspace $(X_{\max},Y_{\min})$.
	\item[2:] Solve WMVSP for $A$ under weights (\ref{eqn:minimal}) 
	to obtain the minimal mv-subspace $(X_{\min},Y_{\max})$.
	\item [3:] Let $A \leftarrow (A^{X_{\max},Y_{\min}^c})^{X^c_{\min}, Y_{\max}}$, which is DM-regular.
	\item [4:] Call q-DM$_{\rm reg}$ for input $A$ to obtain a q-DM chain $\{ (X^k,Y^k) \}_{k}$ for $A$, 
where each $X^k$ (resp. $Y^k$) is viewed as a subspace of a complement of $X_{\min}$ (resp. $Y_{\min}$)
	\item [5:] Return ${\cal C} := \{ (X^k + X_{\min},Y^k + Y_{\min}) \}_{k}$. 
\end{description}

\begin{description}
	\item[Algorithm: q-DM$_{\rm reg}$]
	\item[Input:] A DM-regular partitioned matrix $A$.
	\item[Output:] A q-DM chain ${\cal C}$ of mv-subspaces for $A$.
	\item[1:] For each $\alpha' \in \{1,2,\ldots,\mu\}$, 
	solve WMVSP for weights (\ref{eqn:CD1}), and  
for each $\beta' \in  \{1,2,\ldots,\nu\}$, solve WMVSP for weights~(\ref{eqn:CD2}).
\item [2:] If we find an optimal solution $(X,Y)$ of WMVSP having the weight greater than  that of trivial mv-subspaces, then do the following:
\begin{description}
\item[2.1:]  Call q-DM$_{\rm reg}$ for input $A^{X^c,Y}$ 
to obtain a q-DM chain $\{ (Z^{k}, Y^{k})\}_k$ of $A^{X^c,Y}$.
\item[2.2:]  Call q-DM$_{\rm reg}$ for input $A^{X,Y^c}$ 
to obtain a q-DM chain $\{ (X^{\ell}, W^{\ell})\}_\ell$ of $A^{X,Y^c}$.
\item[2.3:] Return ${\cal C} :=  \{ (Z^k + X, Y^k) \}_k 
\cup \{ (X^\ell, W^\ell + Y) \}_\ell$.
\end{description}
\item[3:] Otherwise, $A$ is quasi DM-irreducible. Return two trivial mv-subspaces.
\end{description}
The correctness of this algorithm follows 
from Lemmas~\ref{lem:recursive}, \ref{lem:minimal_maximal}, \ref{lem:CD} and  Theorem~\ref{thm:q-DM_irreducible}.
The algorithm solves WMVSP polynomially many times.
Since weights $C_{\alpha},D_{\beta}$ are always bounded by a polynomial of $n,m$, 
by Theorem~\ref{thm:main'}, WMVSP can be solved in polynomial time.
Consequently, the whole algorithm runs in polynomial time.
This proves Theorem~\ref{thm:q-DM}.

\section*{Acknowledgments}
We thank Kazuo Murota, Satoru Iwata, Satoru Fujishige, and Yuni Iwamasa 
for helpful comments.
The work was partially supported by JSPS KAKENHI Grant Numbers 25280004, 26330023, 26280004, 17K00029.

\end{document}